\documentclass[a4paper,12pt]{amsart}

\usepackage[T1]{fontenc}
\usepackage[utf8]{inputenc}

\usepackage[english]{babel}

\usepackage{amsmath, amssymb, amsthm, amsfonts}
\usepackage{mathrsfs}
\usepackage{tikz}
\usetikzlibrary{arrows}
\usetikzlibrary{matrix}
\usetikzlibrary{cd}
\usepackage[left=2cm, right=2cm, top=2cm, bottom=2cm]{geometry}
\usepackage{enumerate}
\usepackage{hyperref}
\usepackage{cleveref}

\theoremstyle{definition}
\newtheorem{defi*}{Definition}
\newtheorem{defi}{Definition}[section]
\newtheorem{rema}[defi]{Remark}

\newtheorem*{ques*}{Question}

\theoremstyle{plain}
\newtheorem{thm}[defi]{Theorem}
\newtheorem{lem}[defi]{Lemma}
\newtheorem{cor}[defi]{Corollary}

\newtheorem{prop}[defi]{Proposition}
\newtheorem*{thm*}{Theorem}

\crefname{defi}{Definition}{Definitions}
\crefname{rema}{Remark}{Remarks}
\crefname{exam}{Example}{Examples}
\crefname{thm}{Theorem}{Theorems}
\crefname{lem}{Lemma}{Lemmata}
\crefname{cor}{Corollary}{Corollaries}
\crefname{prop}{Proposition}{Propositions}
\crefname{con}{Conjecture}{Conjectures}
\crefname{ques}{Question}{Questions}
\crefname{section}{Section}{Sections}

\DeclareMathOperator{\Cent}{Cent}
\DeclareMathOperator{\Z}{\mathbb{Z}}
\DeclareMathOperator{\len}{len}
\DeclareMathOperator{\Lat}{Lat}
\DeclareMathOperator{\Soc}{Soc}
\DeclareMathOperator{\Rad}{Rad}

\newcommand{\Ccal}{\mathcal{C}}
\newcommand{\longiso}{\overset{\sim}{\longrightarrow}}

\begin{document}

\title[Decomposition and Structure theorems for modular Garside-like groups]{Decomposition and Structure theorems for Garside-like groups with modular lattice structure}
\author{Carsten Dietzel}
\address{Carsten Dietzel, Department of Mathematics and Data Science, Vrije Universiteit Brussel, Pleinlaan 2, 1050 Brussel, Belgium}
\email{Carsten.Dietzel@vub.be}
\date{\today}

\begin{abstract}
Despite being a vast generalization of Garside groups, right $\ell$-groups with noetherian lattice structure and strong order unit share a lot of the properties of Garside groups. In the present work, we prove that every modular noetherian right $\ell$-group with strong order unit decomposes as a direct product of \emph{beams}, which are sublattices that correspond to the directly indecomposable factors of the strong order interval. Furthermore, we show that the beams of dimension $\delta \geq 4$ can be coordinatized by the $R$-lattices in $Q^{\delta}$, where $Q$ is a noncommutative discrete valuation field with valuation ring $R$. In particular, this gives a precise description of a very big family of modular Garside groups.
\end{abstract}

\maketitle

\section{Introduction}

The theory of Garside group arose out of a work of Garside \cite{garside_braid}, as an approach to an algorithmic solution to the conjugacy problem in the braid groups
\[
B_n = \left\langle \sigma_1,\ldots,\sigma_{n-1} \quad \vline \quad  \begin{matrix}
\sigma_i \sigma_j = \sigma_j \sigma_i & |i - j| > 1 \\
\sigma_i \sigma_{i+1} \sigma_i = \sigma_{i+1} \sigma_i \sigma_{i+1} & 1 \leq i \leq n-2
\end{matrix} \right\rangle
\]

Garside's approach makes use of a special element, nowadays called a \emph{Garside element}, that is - in this case - the least common multiple of all generators $\sigma_i$ ($1 \leq i \leq n-1$) with respect to, say, right-divisibility within the submonoid $B_n^+ \subseteq B_n$ generated by the generators $\sigma_i$.

Garside's methods have soon been generalized by Saito and Brieskorn \cite{brieskorn_saito} to all Artin-Tits groups of spherical type - these groups have the property of being \emph{Garside groups} in the following sense:

\medskip

\textbf{Definition.} \cite[Chapter I, Section 2]{Garside_Foundations}
A \emph{Garside group} is a group $G$ with a distinguished submonoid $G^+$ and a special element $\Delta \in G^+$ - the \emph{Garside element} - such that the following properties hold:
\begin{enumerate}
\item $G^+$ is left- and right-cancellative,
\item $G^+$ is a lattice both under left-divisibility as under right-divisibility
\item there is a length function $\lambda:G^+ \to \Z_0^+$ such that $\lambda(gh) \geq \lambda(g) + \lambda(h)$,
\item in $G^+$, the left- and right-divisors of $\Delta$ coincide and form a finite generating set for the monoid $G^+$
\item $G$ is a group of left fractions for $G^+$, i.e. $G = \{ a^{-1}b : a,b \in G^+ \}$
\end{enumerate}

\medskip

This definition goes back to Dehornoy and Paris \cite{dehornoy_paris}.

All groups of this type are extraordinarily well-approachable by computational methods, as they share the following properties (see \cite{Garside_Foundations}):

\begin{enumerate}
\item All Garside groups have a solvable word problem
\item All Garside groups have a solvable conjugacy problem
\end{enumerate}

Furthermore, the algebraic properties of Garside groups also turn out to be pleasant \cite{charney}:
\begin{enumerate}
\item All Garside groups have a finite-dimensional $K(G,1)$-space.
\item Therefore, each Garside group has finite cohomological dimension.
\end{enumerate}

Despite knowing quite a lot about Garside groups in general, it turns out to be a difficult question which groups are Garside and which are not. Therefore, it (maybe) is a hopeless question to ask which groups exactly \emph{are} Garside groups, meaning that they can be equipped with a Garside structure. We remark here that even the Garside groups that are generated, as groups, by two elements, are not classified yet \cite[p.25, Question 1.]{Garside_Foundations}.

It turns out that a bit more can be said when going over to the divisibility structure, seen as a lattice!

The following result, obtained by Chouraqui and Rump, is the first of this type:

\begin{thm*}
The Garside groups with a distributive lattice structure are exactly the structure groups of finite, non-degenerate, involutive, set-theoretic solutions to the Yang-Baxter equation.
\end{thm*}

Some explanation about the terminology used here:

A \emph{set-theoretic solution} to the \emph{Yang-Baxter equation} is given by a finite set $X$, together with a bijective map $R: X^2 \to X^2$; $R(x,y) =: ({}^xy,x^y)$ - the \emph{$R$-map} - such that the following equation holds for mappings $X^3 \to X^3$:

\[
R_{12}\circ R_{23} \circ R_{12} = R_{23} \circ R_{12} \circ R_{23}.
\]
Here we define $R_{12}(x,y,z) = ({}^xy,x^y,z)$ and $R_{23}(x,y,z) = (x,{}^yz,y^z)$. If the mappings $y \mapsto {}^xy$ are bijective for all $x \in X$ and the mappings $x \mapsto x^y$ are bijective for all $y \in X$, we call the mapping $R$ \emph{non-degenerate}. If $R^2 = \mathrm{id}_{X^2}$, we call the map $R$ \emph{involutive}.

Writing a set-theoretic solution as a pair $(X,R)$, consisting of a set and an $R$-map, we can define a \emph{structure monoid} $M(X,R)$ and a \emph{structure group} as the monoid resp. group generated by the set $X$ under the defining relations $wx = yz$ for all quartuples $w,x,y,z$ with $R(w,x) = (y,z)$.

It turns out that the canonical monoid homomorphism $M(X,R) \to G(X,R)$ is injective and is the Garside monoid of a Garside structure on $G(X,R)$. The respective lattice structure, in turn, is a \emph{distributive} lattice structure. Therefore, we actually \emph{know} distributive Garside groups as well as we know set-theoretic solutions to the Yang-Baxter equation.

Let us now introduce a more general terminology used by Rump in order to get rid of issues related to the finiteness of the set of generators:
\medskip

\textbf{Definition.} \cite{Rump-Geometric-Garside}
A \emph{right $\ell$-group} is a group $G$ that is equipped with a partial order relation $\leq$ that is \emph{right-invariant}, meaning that
\[
x \leq y \Rightarrow xz \leq yz
\]
holds for all $x,y,z \in G$ and furthermore is a lattice under this order.

\medskip

Let $G$ be a right $\ell$-group. An element $s \in G$ is called \emph{normal}, if the equivalence $x \leq y \Leftrightarrow sx \leq sy$ is valid for all $x,y \in G$. An element $s \in G$ is called a \emph{strong order unit} if $s > e$ and for all $g \in G$ there is an integer $n$ such that $s^n \geq g$.

Furthermore, call a right $\ell$-group \emph{noetherian}, if every ascending sequence $g_1 \leq g_2 \leq \ldots$ where all $g_i \leq h$ for some $h \in G$ becomes stationary at some index, and every descending sequence $g_1 \geq g_2 \geq \ldots$ where all $g_i \geq h$ for some $h \in G$ becomes stationary at some index.

The theory of noetherian right $\ell$-groups with strong order unit can be seen as a proper generalization of the theory of Garside groups and is even more fruitful, as these groups can, for example, be used to explain the factorization theory within paraunitary groups \cite{dietzel}, a class of groups that are relevant in signal processing \cite[Chapter 14.]{Paraunitary2}.
Without the noetherianity condition, Garsidean behavior can even be obtained in pure paraunitary groups associated to von Neumann algebras \cite{RD}.

Rump's observation that the Garside groups with distributive lattice structure coincide with certain structure groups of set-theoretic solutions to the Yang-Baxter equation, went hand-in-hand with the development of a \emph{local} theory of \emph{modular} noetherian right $\ell$-groups with strong order unit. He proved that in such a group, a strong order unit $s$ can be chosed in such a way that the strong order interval $[s^{-1},e]$ is a modular \emph{geometric} lattice and obtained a graphical way to describe the relations within $G$ by a labelling of the Hasse-diagram of $[s^{-1},e]$ \cite[Sections 7+8]{Rump-Geometric-Garside}.

In the present work, we attempt to lay the foundations for a \emph{global} structure theory of modular noetherian right $\ell$-groups with strong order unit. In his dissertation, the author proved the following result:
\begin{thm*}
Let $G$ be a modular noetherian right $\ell$-group with strong order unit $s$, such that the strong order interval $[s^{-1},e]$ is a directly indecomposable geometric lattice of length $\delta \geq 4$. Then there exists a noncommutative discrete valuation field $Q$ with valuation ring $R$, such that $G$ is isomorphic, as a lattice, to the lattice of $R$-lattices in $Q^{\delta}$.
\end{thm*}

Here, by an \emph{$R$-lattice}, we mean an essential $R$-submodule, that is finitely generated. From this \emph{coordinatization}, one can furthermore derive an arithmetic description of $G$ as a subgroup of a generalized projective semilinear group.

We continue the analysis of modular, noetherian right $\ell$-groups with strong order unit here. Rump already noted that with a suitably chosen strong order unit $s$, the strong order interval $[s^{-1},e]$ decomposes as a direct product $[s^{-1},e] = \prod_{i=1}^k [z_i,e]$, where each interval $[z_i,e]$ is a directly indecomposable modular geometric lattice.

It was shown by Picantin \cite{picantin} that each Garside group decomposes as an iterated crossed product of Garside groups with cyclic center. This result has been generalized by Rump \cite{rump_garside}, who showed that each noetherian right $\ell$-group admits an appropriate decomposition according to its quasi-center.

For modular noetherian right $\ell$-groups with strong order unit, we produce another decomposition, the \emph{beam decomposition}: this will be a decomposition $G \cong \prod_{i=1}^k \beth(z_i)$, a direct product of lattices associated with the irreducible components $[z_i,e]$ in a direct decomposition of $[s^{-1},e]$ as a lattice. It turns out that a beam $\beth(z_i)$ admits a coordinatization similar to the one presented above as soon as $\delta_{z_i} := \len([z_i,e]) \geq 4$. However, it is notable, that the beams itself are not subgroups but only sublattices of $G$. This has the obvious drawback that we leave the realm of groups by decomposing; the advantage, however, lies in the fact, that for many modular noetherian right $\ell$-groups with strong order unit, the beam decomposition provides an embedding of $G$ into a product of wreath products of semilinear projective groups.

We give an outline of this work:

\begin{itemize}
\item[Section 1.] Introduction
\item[Section 2.] In the preliminaries, we provide basic theorems, notation and definitions about modular lattices (\cref{subs:lattice_theory}) and right $\ell$-groups (\cref{subs:right_l_groups}) that are needed later.
\item[Section 3.] Here, we prove several results regarding factorizations in a (modular) noetherian right $\ell$-group. In particular, we prove the following:

Let $G^- = \{ g \in G : g \leq e \}$ be the \emph{negative cone} of $G$ and $X(G^-) = \{ g \in G^- : g \prec e \}$ the set of elements lying directly underneath $e$, the \emph{dual atoms} of $G^-$. If $G$ is modular, then every factorization of an element $g \in G^-$ as
\[
g = x_k x_{k-1} \ldots x_1
\]
has the same length $k$, so the degree $\deg(g) := k$ is well-defined and can be extended to a homomorphism $\deg : G \to \Z$ (\cref{prop:degree_homomorphism_exists}).

We furthermore give the construction of \emph{right-normal factorizations} in Garside groups in terms of right $\ell$-groups with strong order unit. If $s$ is a strong order unit, then we prove that for all $g \in G^-$ there is a unique factorization
\[
g = g_k g_{k-1} \ldots g_1
\]
with all $g_i \in [s^{-1},e] \setminus \{ e \}$ such that no $g_i$ ($i > 1$) admits a factorization $g_i = h^{\prime}h$ with $h,h^{\prime} \in G^-$, $h \neq e$ and $hg_{i-1} \in [s^{-1},e]$ (\cref{thm:right_normal_facs_exist_and_unique}). Additionally, we prove a useful description of the right-normal factors $g_i$ and also provide a left-handed counterpart (\cref{thm:left_normal_facs_exist_and_unique}).

It turns out that in a modular, noetherian right $\ell$-group with strong order unit, the degree sequence $(\deg(g_i))_{1 \leq i \leq k}$ is nonincreasing (\cref{thm:iota_is_nonincreasing}) when $g = g_k g_{k-1} \ldots g_1$ is a right-normal factorization, which provides a strong tool for understanding factorizations in $G^-$.

Furthermore, we prove that if the strong order interval $[s^{-1},e]$ in a modular noetherian right $\ell$-group is a geometric lattice (which can always be achieved by chosing $s$ suitably, if a strong order unit exists, at all), the meet-irreducible elements in $G^-$ are exactly the elements with a right-normal factorization $g = g_k g_{k-1} \ldots g_1$ such that $\deg(g_i) = 1$ ($1 \leq i \leq k$) (\cref{prop:meet_irreducibles_are_1_homogeneous}). This has the implication that these elements are dual chains (\cref{prop:meet_irreducibles_are chains}).

\item[Section 4.] From here, we will consider modular, noetherian right $\ell$-groups with a modular geometric strong interval $[s^{-1},e]$ only. In this section, we analyse the lattice-theoretic center $\mathrm{Cent}([s^{-1},e]) =: \Ccal(G^-)$ which consists of the uniquely complemented elements $z \in [s^{-1},e]$ that provide a direct product decomposition $[s^{-1},e] \cong [z,e] \times [z^{\prime},e]$. We prove that the center is a sublattice of $[s^{-1},e]$ that is closed under the \emph{arrow-operation} $g \to h := hg^{-1} \vee e$ (\cref{thm:center_closed_under_arrow}).
\item[Section 5.] We show that the dual atoms $z_1,\ldots,z_k$ of the center $\Ccal(G^-)$ generate a subgroup of $G$ that is also a distributive sublattice of $G$ - the \emph{distributive scaffold} of $G$ (\cref{thm:the_distributive_scaffold}).
Therefore, by the Chouraqui-Rump theorem, $G$ contains the structure group of a non-degenerate involutive solution to the Yang-Baxter equation. In order to prove this result, we make use of Rump's \emph{$L$-algebras}.

Afterwards, we construct for each dual atom $z \in X(\Ccal(G^-))$ a strictly decreasing series of elements $\Phi_n(z) \in G^-$ - the \emph{frozen powers} of $z$ - such that the sublattice $\beth^-(z) = \bigcup_{n \geq 1} \Phi_n(z)^{\uparrow}$  - the \emph{semi-beam} associated with $z$ - is a directly irreducible sublattice of $G^-$. The semi-beams provide a factorization of $G^-$ as a product $\prod_{i=1}^k \beth^-(z_i)$, where the $z_i$'s run over all of $X(\Ccal(G^-))$ (\cref{thm:G-_is_product_of_indecomposable_semibeams}). Also, the semi-beams are rigid in the sense that right-multiplication by an arbitrary element in $G^-$ results in a permutation of the set of semi-beams (\cref{prop:rigidity_of_semibeams}).

Finally, we extend the semibeam decomposition $G^- \cong \prod_{i=1}^k \beth^-(z_i)$ to a \emph{beam decomposition} $G \cong \prod_{i=1}^k \beth(z_i)$, where each \emph{beam} $\beth(z_i)$ extends the semibeam $\beth^-(z_i)$ (\cref{thm:G_is_product_of_beams}). We show that the pieces of the beam decomposition are also rigid in the sense that right-multiplication by an arbitrary element of $G$ results in a permutation of the beams (\cref{thm:rigidity_of_beams}).

\item[Section 6.] In this section, we associate with each beam $\beth(z)$ its \emph{dimension}, which is the integer $\delta_z := \deg(z)$. We prove that all beams $\beth(z)$ with $\delta_z \geq 4$ are \emph{coordinatizable} in the sense that there exists a discrete valuation field $Q_z$ with valuation ring $R_z$ such that there is an isomorphism of lattices $\beth(z) \cong \Lat({}_{R_z}Q_z^{\delta_z})$, where $\Lat({}_RQ^{\delta})$ is, as described above, defined as the lattice of all $R$-lattices in $Q^{\delta}$ (\cref{thm:large_enough_beams_are_coordinatizable}).

In order to obtain this result, we show that each $\beth(z)$ is a primary lattice (\cref{prop:beams_are_primary}) and construct a piecewise coordinatization of the intervals $[\Phi_n(z),e]$ by submodule lattices of modules over completely primary uniserial rings, where we make use of Inaba's coordinatization theorem (\cref{prop:piecewise_coordinatization}). We then show that these single coordinatizations can be put together to a coordinatization of $\beth^-(z)$ by $\Lat({}_{R_z}R_z^{\delta_z}) =: \Lat(R_z,\delta_z)$, where $R_z$ is a noncommutative discrete valuation ring (\cref{thm:large_enough_beams_are_coordinatizable}).

We conclude the proof by extending this coordinatization to a lattice isomorphism $\beth(z) \cong \Lat({}_{R_z}Q_z^{\delta_z})$.
\item[Section 7.] In the last section, we briefly demonstrate the implications of the preceding results on the group-theoretic structure of modular noetherian right $\ell$-groups with strong order unit.
\end{itemize}

We emphasize that there is \emph{some} textual overlap with the author's dissertation \cite{dietzel_dissertation}. These are results that have not appeared anywhere else, yet.

\section{Preliminaries}

\subsection{Lattice theory} \label{subs:lattice_theory}

For the basic theory of lattices, see \cite{Graetzer-Lattice}.

In what follows, let $(L, \wedge, \vee)$ be a lattice with its \emph{meet}- and \emph{join}-operation denoted by $\wedge$ resp. $\vee$. Furthermore, when referring to the partial order relation $\leq$ on $L$, we will always (except when explicitly stated otherwise) mean the one defined by $a \leq b : \Leftrightarrow a \wedge b = a \Leftrightarrow a \vee b = b$. In most cases, we will just refer to a set $L$ as a lattice and implicitly assume that it comes with the respective operations $\wedge, \vee$.

We will occasionally use meets and joins over empty or possibly non-finite subsets of lattices: if $L$ is a lattice and $A \subseteq L$ a subset thereof, we define the \emph{meet} of $A$ as the unique element $z \in L$ - if existent - such that
\[
\forall x \in L:  \quad (x \leq z) \Leftrightarrow \left( \forall y \in A: x \leq y \right). 
\]
This element will be denoted as $z =: \bigwedge A$. Dually, we defined the \emph{join} of $A$ as the unique element $z \in L$ - if existent - such that
\[
\forall x \in L:  \quad (x \geq z) \Leftrightarrow \left( \forall y \in A: x \geq y \right). 
\]
If $0_L := \bigvee \emptyset$ resp.  $1_L := \bigwedge \emptyset$ exists in $L$, we call $L$ \emph{bounded from below} resp. \emph{bounded from above}. When both of these elements exist, then we plainly call $L$ \emph{bounded}.

For $a,b \in L$, we denote $[a,b] = \{ x \in L : a \leq x \leq b \}$ and $a^{\uparrow} = \{ x \in L : x \geq a \}$.

If $L_1,L_2$ are lattices, a map $\varphi: L_1 \to L_2$ is called a \emph{homomorphism} (of bounded lattices) if it respects the meet and join operations (and $0_L,1_L$) of $L_1,L_2$ in the obvious way. If this homomorphism is bijective, it is called an \emph{isomorphism}. Clearly, the inverse of an isomorphism of lattices is an isomorphism as well. 

If $L$ is a lattice, a subset $C \subseteq L$ is called a \emph{chain} whenever $x \leq y$ or $y \leq x$ holds for any $x,y \in C$. For a chain $C$, the quantity $\len(C) = |C|-1 \in \Z_{\geq -1} \cup \{ + \infty \}$ is called its \emph{length}. Furthermore, for a lattice $L$, its length is defined as
\[
\len(L) := \sup \{ |C|-1 : C \subseteq L \textnormal{ is a chain} \}.
\]

If $L$ is a lattice, one says that $L$ fulfills the \emph{ascending chain condition} if every infinite ascending sequence $x_1 \leq x_2 \leq \ldots$ with all $x_i \in L$ becomes stationary at some point; on the other side, $L$ fulfills the \emph{descending chain condition} if every infinite descending sequence $x_1 \geq x_2 \geq \ldots$ becomes stationary at some point.

In this article, will mainly be concerned about \emph{modular} and \emph{distributive} lattices:

\begin{defi}
\begin{enumerate}[1)]
\item A lattice $L$ is called \emph{modular}, if for all $a,b,c \in L$ with $a \leq c$ the equation $a \vee (b \wedge c) = (a \vee b) \wedge c$ is valid.
\item A lattice $L$ is called \emph{distributive}, if for all $a,b,c \in L$ the equation $a \vee (b \wedge c) = (a \vee b) \wedge (a \vee c)$ is valid.
\end{enumerate}
\end{defi}

It is well-known that modular lattices make up an equational class and that the distributive lattices are a proper subclass thereof.

An important property characterizing modular lattices is the following one:

\begin{lem}[\textbf{Diamond isomorphism lemma}] \label{lem:diamond_lemma}
A lattice $L$ is modular, if and only if for all $a,b \in L$, the map
\begin{align*}
\varphi_b : [a, a \vee b] & \to [a \wedge b, b] \\
x & \mapsto x \wedge b
\end{align*}
is an isomorphism of lattices. If this is the case, the inverse map is given by $\psi_a(y) = a \vee y$.
\end{lem}

\begin{proof}
See \cite[Theorem 348]{Graetzer-Lattice}.
\end{proof}

Given two elements $x,y$ in a lattice $L$, we write $x \prec y$ if $x < y$ and if there is no element $z \in L$ such that $x < z < y$. If the lattice $L$ is bounded from above with top element $1_L$, we write
\[
X(L) := \{ x \in L : x \prec 1_L \}
\]
and call the elements of this set the \emph{dual atoms} of $L$. For convenience, we write $\tilde{X}(L) = X(L) \cup \{ 1_L \}$.

An element $x \in L$ is called \emph{meet-irreducible} if there is no representation $x = a \wedge b$ with $a,b \neq x$. Dually, $x$ is called \emph{join-irreducible} if no representation $x = a \vee b$ with $a,b \neq x$ exists. Clearly, each dual atom is meet-irreducible.

A central role will be played by (dually) geometric lattices. For the sake of simplicity, we will specialize the definition to bounded modular lattices, which are the only geometric lattices used in this article.

\begin{defi}[\cite{Graetzer-Lattice}]
A modular lattice $L$ of finite length is called \emph{geometric} if it is \emph{atomistic}, meaning that each element $a \in L$ is of the form $a = \bigvee A$ for some - possibly empty - finite subset $A \subseteq L$ where $x \succ 0$ for all $x \in L$.
\end{defi}

For the sake of completeness we constrast this definition with its dual notion, which explicitly involves the relevant set of dual atoms $X(L)$.

\begin{defi}
A modular lattice $L$ of finite length is called \emph{dually geometric} if it is \emph{dually atomistic}, meaning that each element $a \in L$ is of the form $a = \bigwedge A$ for some - possibly empty - finite subset $A \subseteq X(L)$.
\end{defi}

Note that for a modular lattice of finite length, the notions of geometricity and dual geometricity actually coincide; this can, for example, be seen from the characterization of directly irreducible geometric lattices \cite[Chapter V, Section 5]{Graetzer-Lattice}.

\begin{rema}
The reader may have already noted our emphasis on dual notions (with regard to atoms, frames, bases, geometricity, \ldots). This is due to the fact that this work builds on results from \cite{Rump-Geometric-Garside} and \cite{dietzel_dissertation}, where heavy use is made of \emph{$L$-algebras}; as these objects are derived from algebraic logic, their \emph{logical unit} - which coincides with the identity in the case of cones in right $\ell$-groups - should be the top element.

Also the reader may convince himself that it makes more sense for us to restrict to bounded modular lattices, as all geometric lattices that turn up in this work appear as strong order intervals in modular right $\ell$-groups, which are necessarily bounded and modular.
\end{rema}

\begin{defi}
A bounded lattice $L$ is called \emph{complemented}, if for each $x \in L$ there is an element $x^{\prime}$ - its \emph{complement} - such that $x \vee x^{\prime} = 1_L$ and $x \wedge x^{\prime} = 0_L$.
\end{defi}

An important property of dually geometric lattices is the following one:

\begin{prop}
Each dually geometric lattice is complemented.
\end{prop}

\begin{proof}
As the property of being complemented is clearly self-dual, this is just \cite[Theorem 392]{Graetzer-Lattice}.
\end{proof}

\begin{defi}
Let $L$ be a bounded distributive lattice. Then $L$ is called \emph{Boolean} if every element is uniquely complemented.
\end{defi}

\begin{defi}

Let $L$ be a bounded lattice. Then a subset $X \subseteq L$ is called \emph{(dually) independent} if for any subsets $A,B \subseteq X$, we have  $\left( \bigvee A \right) \wedge \left( \bigvee B \right) = \bigvee (A \cap B)$ (resp. $\left( \bigwedge A \right) \vee \left( \bigwedge B \right) = \bigwedge (A \cup B)$).

A subset $X \subseteq L$ is called \emph{(dually) spanning} if $\bigvee X = 1_L$ (resp. $\bigwedge X = 0_L$).

A subset $X \subseteq L$ is called a \emph{(dual) frame}, if $X$ is (dually) independent and (dually) spanning and a \emph{(dual)} \emph{basis}, if every element of $X$ is meet-irreducible (resp. join-irreducible).
\end{defi}

For a bounded lattice $L$, call an element $x \in L$ a \emph{chain} (resp. \emph{dual chain}) if the interval $[0_L,x]$ (resp. $[x,1_L]$) is a chain.

We we will later need the following easy lemma.

\begin{lem} \label{lem:reflection_of_frames}
Let $L$ be a bounded modular lattice. Then $L$ has a frame $x_1, \ldots , x_{\delta}$ if and only if $L$ has a dual frame $x_1^{\prime}, \ldots, x_{\delta}^{\prime}$. In particular, the frames can be chosen in such a way that $[0_L,x_i] \cong [x_i^{\prime},1_L]$.
\end{lem}

\begin{proof}
Assume that $L$ has the frame $X = \{ x_1,\ldots , x_{\delta} \}$. We construct the elements
\[
x_i^{\prime} = \bigvee (X \setminus \{ x_i \})
\]
Using the diamond lemma \cref{lem:diamond_lemma}, we see that
\[
[x_i^{\prime},1_L] = [x_i^{\prime}, x_i^{\prime} \wedge x_i] \cong [x_i^{\prime} \vee x_i, x_i] = [0_L,x_i].
\]
We set $X^{\prime} = \{ x_1^{\prime}, \ldots , x_{\delta}^{\prime} \}$. For $A \subseteq X$, write $A^{\prime} = \{ x_i^{\prime} : x_i \in A \}$. Under this notation, we see that
\begin{align*}
\bigwedge A^{\prime} & = \bigwedge_{x_i \in A} \left( \bigcap (X \setminus \{ x_i \}) \right) \\
& = \bigvee \left( \bigwedge_{x_i \in A} (X \setminus \{ x_i \}) \right) \\
& = \bigvee (X \setminus  A ).
\end{align*}
Therefore, we can calculate for arbitrary $A^{\prime},B^{\prime} \subseteq X^{\prime}$:
\begin{align*}
\left( \bigwedge A \right) \vee \left( \bigwedge B \right) & = \left(\bigvee (X \setminus A) \right) \vee \left( \bigvee (X \setminus B) \right) \\
& = \bigvee \left( X \setminus (A \cap B) \right) \\
& = \bigwedge (A \cap B)^{\prime} \\
& = \bigwedge (A^{\prime} \cap B^{\prime}). 
\end{align*}
Therefore, the elements of $X^{\prime}$ are dually independent. They are also dually spanning, as $\bigwedge X^{\prime} = \bigvee (X \setminus X) = \bigvee \emptyset = 0_L$.

The reverse direction follows by duality.
\end{proof}

If $L$ is a modular geometric lattice, then the join- resp. meet- irreducible elements of $L$ coincide with the atoms resp. dual atoms in $L$. We define the \emph{dimension} of $L$ as the size of a basis for $L$; note that this is an invariant, as the dimension of $L$ must then coincide with $\len(L)$.

By \cref{lem:reflection_of_frames}, we easily see:

\begin{prop} \label{prop:dimension_of_geometric_lattice_is_size of_dual_basis}
Let $L$ be a modular geometric lattice. Then the dimension of $L$ coincides with the size of an arbitrary dual basis of $L$.
\end{prop}

We assume that the reader is familiar with the notion of a direct product of lattices, but we recall it nevertheless, as it will later be necessary to distinguish between \emph{external} and \emph{internal} direct products of lattices.

We recall the definition of an external direct product:

Let $(L_i)_{i \in I}$ be a collection of lattices; then the \emph{external direct product} is given by the cartesian product $\prod_{i \in I} L_i$ which becomes a lattice under the coordinatewise order or, which is equivalent, under the coordinatewise lattice operations.

Accordingly, an \emph{external decomposition} of a lattice $L$ is simply given by an isomorphism $\varphi: L \overset{\sim}{\to} \prod_{i \in I} L_i$ where $(L_i)_{i \in I}$ is a family of nontrivial lattices.

We will later also need the notion of a finitary \emph{internal} product of bounded lattices that the author has not found in explicit form in the standard literature.

If $L$ is a lattice and $S \subseteq L$ a subset thereof, we call $S$ \emph{upwards-closed} if $x \in S$ and $y \geq x$ imply $y \in S$.

\begin{defi}
Let $L$ be a lattice that is bounded from above. Let $L_1,\ldots,L_k \subseteq L$ be sublattices that are upwards closed. We call $L$ an \emph{internal product} of the lattices $L_i$ ($1 \leq i \leq k$) if every element $x \in L$ can \emph{uniquely}	be presented as $x = \bigwedge_{i=1}^k x_i$ with $x_i \in L_i$.
\end{defi}

Note that each collection of elements $(x_i)_{1 \leq i \leq k}$ with $x_i \in L_i$ is necessarily dually independent, when the $L_i$ are the factors of an internal product of lattices.

It will be important later that, in case of a bounded from above lattice, internal products are equivalent to external products:

\begin{prop} \label{prop:external_products_are_internal_products}
Let $L,L_1, \ldots , L_k$ be lattices that are bounded from above. Then there is an isomorphism $\varphi : L \to \prod_{i = 1}^kL_i$ if and only if $L$ is an internal product of sublattices $L_i^{\prime} \subseteq L$ with $L_i^{\prime} \cong L_i$.  
\end{prop}

\begin{proof}
Note that under an isomorphism $\varphi: L \overset{\sim}{\to} \prod_{i =1}^kL_i$, a factor $L_i$ can be identified with the sublattice
\[
L_i^{\prime} := \varphi^{-1} \left( \prod_{j=1}^{i-1}\{ 1_{L_j} \} \times L_i \times \prod_{j=i+1}^{k}\{ 1_{L_j} \} \right),
\]
therefore, an isomorphism with an external product always implies an internal product decomposition.

On the other hand, let $L = \prod_{i=1}^k L_i$ be an internal product decomposition; then the map $\psi: \prod_{i=1}^k L_i \to L ; (x_i)_{1 \leq i \leq k} \mapsto \bigwedge x_i$ is clearly order-preserving and a bijection, the latter fact coming from the definition of an internal direct product.

We claim that $\varphi: L \to \prod_{i=1}^kL_i$ where $\varphi(x) = (x_i)_{1 \leq i \leq k}$ with $x_i = \min (x^{\uparrow} \cap L_i)$ is well-defined and inverse to $\psi$:

Let $x \in L$ and write it as $x = x_1^{\prime} \wedge \ldots x_k^{\prime}$ with $x_i^{\prime} \in L_i$. For an arbitrary index $i$, take an element $y \in L_i$ with $x \leq y$, then
\[
x = x_1^{\prime} \wedge \ldots \wedge (x_i^{\prime} \wedge y) \wedge \ldots x_k^{\prime}.
\]
By uniqueness, this shows $x_i^{\prime} \wedge y = x_i^{\prime}$, that is, $x_i^{\prime} \leq y$. Therefore, $\varphi$ is indeed the inverse to $\psi$. As $\varphi$ and $\psi$ are bijections that are easily seen to be order-preserving, they are isomorphisms of lattices.
\end{proof}

If $L$ is the internal direct product of upsets $z_1^{\uparrow},\ldots, z_k^{\uparrow} \subseteq L$, the above proof shows that the respective isomorphism can be written as:
\begin{align*}
\varphi: L & \to \prod_{i=1}^k z_i^{\uparrow} ; \\
x & \mapsto (x \vee z_i)_{1 \leq i \leq k}.
\end{align*}

If it is clear from the context, if a direct decomposition of a lattice is internal or external, or if it does not matter (because of \cref{prop:external_products_are_internal_products}), we will often not emphasize the \emph{type} of product that is considered.

Let the lattice $K$ be bounded from above and assume that we are given an internal decomposition $K = \prod_{i=1}^kK_i$, where each piece $K_i$ is directly indecomposable, i.e. cannot be decomposed as a direct product of two non-trivial sublattices. We will later face the situation that $K$ is a downset in a bigger lattice, i.e. that $K \subseteq L$ is such that $K = x^{\downarrow}$ for some $x \in L$.

Assume we are given an internal decomposition $L = \prod_{i=1}^kL_i$ and we have a sublattice $K = 1_K^{\downarrow}$ for some element $1_K \in L$, we can write $1_K = \bigwedge_{i=1}^k \varepsilon_i^{L,K}$ for some uniquely determined elements $\varepsilon_i^{L,K} \in K_i$.

Therefore, we have an external (!) decomposition $K \cong \prod_{i=1}^k (\varepsilon_i^{L,K} )^{\downarrow}$ which is, however, not internal, as the factors in the decomposition potentially lie in $L \setminus K$. As $K$ is a bounded lattice, there is of course a canonical way of associating an internal decomposition with the external one that is forced upon by $L$. However, we will encounter the following problem:

\medskip

\textbf{Problem.} \textit{Given a chain $L^{(1)} \subseteq L^{(2)} \subseteq \ldots$ of lattices, whose union $L^{(\infty)}$ can be unbounded, assume that each lattice $L^{(n)}$ has an internal factorization $L^{(n)} = \prod_{i=1}^k L_i^{(n)}$, for fixed $k$, that induces the respective factorizations of $L^{(m)}$ ($m < n$). Find an external (!) factorization $L^{(\infty)} = \prod_{i=1}^k L_i^{(\infty)}$, where each factor $L_i^{(\infty)} \subseteq L^{(\infty)}$ is a convex sublattice such that $L_i^{(\infty)} \cap L^{(1)} = L_i^{(1)}$ for all factors.}
\medskip

We will use the following construction to \emph{extend} factorizations of $L^{(1)}$ to $L^{(n)}$ in a compatible way by a shifting procedure.

For this sake, let $K \subseteq L$ be two lattices such that $K$ is a downset in $L$. Assume that we are given an internal decomposition $L = \prod_{i=1}^k L_i$. Writing $1_K = \bigwedge_{i=1}^k \varepsilon_i^{L,K}$, we clearly have an external decomposition $K \cong \prod_{i=1}^k (\varepsilon^{L,K}_i)^{\downarrow} =: \prod_{i=1}^k K_i$.

Now define elements $\overline{\varepsilon}_i^{L,K} = \bigwedge_{j \neq i} \varepsilon_j^{L,K}$.

We now make the following definitions: given a factor $K_i \subseteq K$, we define its \emph{extension to $L$} as the sublattice
\[
K_i^L = \{ \overline{\varepsilon}_i^{L,K} \wedge y : y \in L_i \} \subseteq L.
\]
Representing $L$ externally as $\prod_{i = 1}^k L_i$, it becomes clear that $K_i^L$ is represented by the $k$-tuples $(x_j)_{1 \leq j \leq k}$ where $x_j = \varepsilon_j^{L,K}$ for $j \neq i$, so $K_i^L$ is a sublattice of $L$ that is isomorphic to $L_i$. It will become clear soon that $K_i \subseteq K_i^L$

Let $L = \prod_{i=1}^k L_i$ be an internal decomposition and let $\varphi: L \overset{\sim}{\to} \prod_{i=1}^k L_i$ be the corresponding isomorphism with an external direct product, with coordinates $\varphi_i(x) = x_i^L$ whenever $x = \bigwedge_{i =1}^k x_i^L$ is the (unique) representation as a meet of elements $x_i^L \in L_i$. We fix the notation $x_i^L$ for the $L_i$-component of a meet-representation of $x$ with regard to $(L_i)_{1 \leq i \leq k}$.

As $\varphi$ is an isomorphism and the maps $L_i \overset{\sim}{\to} K_i^L$ given by $x \mapsto x \wedge \overline{\varepsilon}_i^{L,K}$ are also isomorphisms which can be seen best within the external decomposition of $L$, we can now define the isomorphism:
\begin{align*}
\psi^{L,K} : L  & \overset{\sim}{\to} \prod_{i=1}^k K_i^L \\
x & \mapsto \left( x_i^L \wedge \overline{\varepsilon}_i^{L,K} \right)_{1 \leq i \leq k}. 
\end{align*}

However, it is not clear, yet, if these isomorphisms are compatible with each other. It turns out that they indeed are, as is shown by the following lemma:

\begin{lem} \label{lem:factor_extensions_are_compatible}
Let $K \subseteq L \subseteq M$ be lattices, where each inclusion embeds a lattice as a downset into the next one. Then for all $x \in L$ we have the equality:
\[
\psi^{L,K}(x) = \psi^{M,K}(x).
\]
In particular, $K_i \subseteq K_i^L$.
\end{lem}

\begin{proof}
We have the coordinates $\psi_i^{L,K}(x) = x_i^L \wedge \overline{\varepsilon}_i^{L,K}$ and $\psi_i^{M,K}(x) = x_i^M \wedge \overline{\varepsilon}_i^{M,K}$.

In the product representation of $M$ we see that $x_i^M \leq \varepsilon_i^{M,L}$ and $\varepsilon_i^{M,K} \leq \varepsilon_i^{M,L}$. Also we see that $x_i^L = x_i^M \wedge 1_L$ and, similarly $\varepsilon_i^{L,K} = \varepsilon_i^{M,K} \wedge 1_L$. Using this, we calculate:
\begin{align*}
\psi_i^{L,K}(x) & = x_i^L \wedge \overline{\varepsilon}_i^{L,K} \\
& = x_i^M \wedge 1_L \wedge \overline{\varepsilon}_i^{M,K} \\
& = x_i^M \wedge \varepsilon_i^{M,L} \wedge \overline{\varepsilon}_i^{M,L} \wedge  \overline{\varepsilon}_i^{M,K} \\
& = x_i^M \wedge \overline{\varepsilon}_i^{M,K} = \psi_i^{M,K}(x).
\end{align*}
The fact that $K_i \subseteq K_i^L$ now follows after replacing $L$ by $K$ and $M$ by $L$.
\end{proof}

\begin{lem} \label{lem:restriction_of_factor_extension}
Let $K \subseteq L$ be an inclusion of a lattice $K$ as a downset of a lattice $L$, with a factorization $K = \prod_{i=1}^k K_i$ that is induced from a factorization $L = \prod_{i=1}^k L_i$. Then $K_i^L \cap K = K_i$ for all $1 \leq i \leq k$.
\end{lem}

\begin{proof}
The elements of $K_i^L$ are represented within $\prod_{i=1}^k K_i^L$ as $(x_j)_{1 \leq j \leq k}$ where $x_j = 1_K$ for $j \neq i$. Therefore, any element in $x \in K_i^L \cap K$ is represented by a tuple $\psi^{L,K}(x) = (x_j)_{1 \leq j \leq k}$ with $x_j = 1_K$ for $j \neq i$ and $x_i \in K_i$ which implies $\psi(x) = x_i \in K$.
\end{proof}

\begin{prop} \label{prop:glueing_factorizations}
Assume that we are given a lattice $L^{(\infty)} = \bigcup_{n \geq 1} L^{(n)}$, such that  $\left( L^{(n)} \right)_{n \geq 1}$ is an increasing chain of principal ideals in $L^{(\infty)}$ - meaning that $L^{(n)} = x_n^{\downarrow}$ for some $x_n \in L^{(\infty)}$. Furthermore, assume that each ideal internally factorizes as
\[
L^{(n)} = \prod_{i=1}^k L_i^{(n)}
\]
in a way that the factorization of $L^{(m)}$ ($m < n$) is the one induced from $L^{(n)}$. Then there exists an external factorization $L^{(\infty)} = \prod_{i = 1}^k L_i$, where $L_i \subseteq L^{(\infty)}$ are sublattices such that $L_i \cap L^{(1)} = L_i^{(1)}$.
\end{prop}

\begin{proof}
For $j \geq i$, abbreviate $L_i^n := \left( L_i^{(1)} \right) ^{L_n}$. We define the sublattices $L_i = \bigcup_{n \geq 1} L_i^n \subseteq L$.

By \cref{lem:factor_extensions_are_compatible}, we have a well-defined homomorphism $\psi: L^{(\infty)} \to \prod_{i=1}^k L_i$ given by $\psi(x) = \psi^{L^{(n)},L^{(1)}}(x)$ if $x \in L^{(n)}$. As the restriction to $L^{(n)}$ is the isomorphism $\psi^{L^{(n)},L^{(1)}}: L^{(n)} \overset{\sim}{\to} \prod_{i=1}^k L_i^n$, the map $\psi$ is an isomorphism.

$L_i \cap L^{(1)} = L_i^{(1)}$ follows from \cref{lem:restriction_of_factor_extension}.
\end{proof}

\subsection{Right \texorpdfstring{$\ell$}{l}-groups} \label{subs:right_l_groups}

A \emph{right-ordered group} is a pair $(G,\leq)$ where $G$ is a group and $\leq$ a partial order relation on $G$ that is \emph{right-invariant}, meaning that for all $x,y,z \in G$, we have the implication
\[
x \leq y \Rightarrow xz \leq yz.
\]
In fact, this is an equivalence. If in a right-ordered group $G$ the partial order $\leq$ comes from a lattice structure on $G$, we will call $G$ a \emph{right lattice-ordered group}, or \emph{right $\ell$-group}, for short. It is easily seen that in a right $\ell$-group, the lattice operations are right-invariant as well, meaning that
\[
(x \wedge y)z = xz \wedge yz \quad ; \quad (x \vee y)z = xz \vee yz
\]
for any $x,y,z \in G$.

As usual, when talking about a right-ordered group or a right $\ell$-group, we will in most cases just refer to the set $G$; the relations resp. operations are then implicit.

Two important subsets in a right-ordered group are the \emph{negative cone} $G^- := \{ g \in G : g \leq e \}$ and the \emph{positive cone} $G^+ := \{ g \in G : g \geq e \}$. Note that both $G^+$ and $G^-$ are closed under multiplication and contain the neutral element $e \in G$. Therefore, they can be considered as submonoids of $G$. An important property of the negative cone is its \emph{purity}, which means that:

\begin{equation}
\forall x \in G: \left((x \in G^-) \ \& \ (x^{-1} \in G^-) \right) \Rightarrow (x = e). \label{eq:purity_of_cone}
\end{equation}
On the other hand, each pure submonoid of a group gives a unique right-invariant order on the group, as, by right-invariance, the partial order $\leq$ is determined by $G^-$ through $g \leq h \Leftrightarrow gh^{-1} \leq e \Leftrightarrow gh^{-1} \in G^-$ \cite[Theorem 1.5.1.]{kopytov}.\footnote{Note that we emphasize the role of the \emph{negative} cone instead of the \emph{positive} cone here. The reason for this is that we will often make use of the $L$-algebra structure of $G^-$. As an $L$-algebra is a logical algebra, it is bounded from above by an element that algebraically represents tautological truth.}

From the given order $\leq$ we can also construct the \emph{dual order} $\lesssim$ on $G$, defined by $g \lesssim h : \Leftrightarrow h^{-1} \leq g^{-1}$. This is easily seen to be the unique \emph{left}-invariant order on $G$ that has the same negative (i.e. $\leq e$) elements as $\leq$. If $(G,\leq)$ is a right $\ell$-group, then $(G,\lesssim)$ is also lattice-ordered with lattice operations
\[
g \curlyvee h = (g^{-1} \wedge h^{-1})^{-1} \quad ; \quad g \curlywedge h = (g^{-1} \vee h^{-1})^{-1}.
\]
The order relation and operations are left-invariant, which amounts to saying that $(G^{\mathrm{op}}, \lesssim)$ - $G^{\mathrm{op}}$ being the group $G$ with the \emph{opposite} multiplication $g \cdot_{\mathrm{op}} h := hg$ - is a \emph{right} $\ell$-group again.

In itself, a right $\ell$-group already has some important Garsidean properties, namely that elements can be written as a right- resp. left fractions of elements in $G^-$:

\begin{prop} \label{prop:group_of_fractions}
A right $\ell$-group $G$ is a group of left fractions for $G^-$, that is, every $g \in G$ can be written as $g = g_1^{-1}g_2$ with $g_1, g_2 \in G^-$. Similarly, $G$ is a group of right fractions for $G^-$.
\end{prop}

\begin{proof}
Let $g \in G$. We can express $g$ as a left fraction $g_1^{-1}g_2$ with $g_2 := g \wedge e \leq e$ and $g_1 := (gg_2^{-1})^{-1} \leq e$. A possible right fraction $g = g_1^{\prime -1}g_2^{\prime}$ is given by $g_2^{\prime} := (g \vee e)^{-1} \leq e$ and $g_1^{\prime} := gg_2^{-1} \leq e$.
\end{proof}

The following - regularly useful - result essentially says that the relations in $G$ are determined by those in $G^-$:

\begin{prop} \label{prop:extending_homomorphisms}
If $G$ is a right $\ell$-group and $H$ is an arbitrary group, then any monoid homomorphism $f^{\prime}: G^- \to H$ can uniquely be extended to a group homomorphism $f: G \to H$ with $f|_{G^-} = f^{\prime}$ by setting $f(g_2^{-1}g_1) := f^{\prime}(g_2)^{-1} f^{\prime}(g_1)$ for all $g_1,g_2 \in G^-$.
\end{prop}

\begin{proof}
We only give a sketch of the proof: it is easily seen that $G^-$ and $H$ are \emph{left-Ore monoids}, meaning that they are left-cancellative and any pair of elements has a common left-multiple. For both monoids, the first property follows from being submonoids of a group; the second property for $G^-$ follows from the existence of meets in $G^-$, whereas for the group $H$, it is trivially fulfilled. Now the statement can be immediately derived from \cite[Chapter II, Lemma 3.13]{Garside_Foundations}.
\end{proof}

An operation that will be important later is the $\to$-operation, which makes $G^-$ into a residuated monoid, a special kind of \emph{$L$-algebra}, an algebraic structure introduced by Rump \cite{self_similarity}. This operation is defined as:
\begin{align*}
\to : G^- \times G^- & \to G^- \\
g \to h & := hg^{-1} \wedge e = (h \wedge g)g^{-1}.
\end{align*}

As the $\to$-operation is often used in calculations involving group multiplication as well, we use the following convention:

\textbf{Convention: } \textit{In the absence of brackets, multiplications are evaluated before arrows!}

We will often make use of the following identities for the $\to$-operation.

\begin{lem} \label{lem:equations_for_arrow}
For all $x,y,z \in G^-$ we have:
\begin{align}
x \rightarrow x & = x \rightarrow e = e \tag{S1a} \label{eq:s1a} \\
e \rightarrow x & = x \tag{S1b} \label{eq:s1b} \\
(x \rightarrow y) x & = x \wedge y \label{eq:s2} \tag{S2} \\
(x \wedge y) \rightarrow z & = (x \rightarrow y) \rightarrow (x \rightarrow z) \label{eq:s3} \tag{S3} \\
x \rightarrow (y \wedge z) & = (x \rightarrow y) \wedge (x \rightarrow z) \tag{S4} \label{eq:l_monotone} \\
xy \rightarrow z & = x \rightarrow(y \rightarrow z) \tag{S5} \label{eq:l_structure_left} \\
x \rightarrow yz & = ((z \rightarrow x) \rightarrow y)(x \rightarrow z) \tag{S6} \label{eq:l_structure_right} \\
x \rightarrow y = e & \Leftrightarrow x \leq y. \tag{S7} \label{eq:l_inequality}
\end{align}
\end{lem}

\begin{proof}
Essentially, these statements are all trivial or well-known. For a worked-out proof of all of these statements, see \cite[Proposition 2.1.11.]{dietzel_dissertation}, for example.
\end{proof}

\section{Factorization theory} \label{sec:factorization_theory}

A large part of the following material, including most proofs, can also be found in the dissertation of the author (\cite{dietzel_dissertation}); however, these results have not been published anywhere else. Also, several basic notions that we work with - such as the concept of strong order units in the case of right $\ell$-groups - are due to Rump (\cite{Rump-Geometric-Garside}).

Let $G$ be a right $\ell$-group. Furthermore, let $s \in G$. One calls $s$ \emph{normal} if left-multiplication by $s$ is a lattice-isomorphism, that is, if for all $x,y \in G$ we have
\[
s(x \wedge y) = sx \wedge sy \quad , \quad s(x \vee y) = sx \vee sy.
\]
Furthermore, if $s > e$ is normal and every $g \in G$ is dominated by some power of $s$ in the sense that $g \leq s^k$ for some $k \in \Z$, the element $s$ is called a \emph{strong order unit}. By normality, this easily implies that every $g \in G$ also dominates some power of $g$, i.e. $g \geq s^l$ for some $l \in \Z$.

We note, for later, the following useful observation:

\begin{lem} \label{lem:conjugation_by_s_is_automorphism}
Let $s \in G^+$ be normal, then the map $[s^{-1},e] \to [s^{-1},e]; \ x \mapsto s^{-1}xs$ is an isomorphism of lattices.
\end{lem}

Note that although we are not allowed to multiply arbitrary inequalities in a group with only right-invariant order, we can do so when normal elements are involved:

\begin{lem} \label{lem:multipl_inequalities_with_normal_elements}
If $s_1,s_2$ are normal elements in $G$ and $g_1,g_2 \in G$, then $g_i \geq s_i$ for $i \in \{1,2 \}$ implies $g_1g_2 \geq s_1s_2$.
\end{lem}

\begin{proof}
Under the given conditions, $s_1s_2 \leq s_1g_2 \leq g_1g_2$.
\end{proof}

Even in absence of finiteness or chain conditions, a strong order unit has enough of the properties of a Garside element, so we can work with certain Garside-theoretical concepts.

One of the most important ones is the existence of right-normal factorizations in $G^-$:

\begin{defi}
Let $s$ be a strong order unit in $G$, and let $g_1,\ldots,g_k$ be a finite sequence of elements in $[s^{-1},e]$ for some $k \in \Z_{\geq 0}$. For an element $g \in G^-$, we call the sequence $g_1,\ldots,g_k$ a \emph{right-normal factorization} of $g$ if:
\begin{enumerate}[1)]
\item $g = g_k\ldots g_2g_1$, i.e. the sequence is indeed a factorization of $g$,
\item $g_i \neq e$ for all $1 \leq i \leq k$, and
\item for all $1 \leq i < k$, there do \emph{not} exist $h,h' \in G^-$ with $h \neq e$, $h'h = g_{i+1}$ and $hg_i \in [s^{-1},e]$ (\emph{right-maximality}).
\end{enumerate}
\end{defi}
Note that right-maximality of the factorization means that no factor can be \glqq enlarged\grqq\ with a split off non-trivial right-divisor of its left neighbor. That means that $s^{-1}g_i^{-1}$ and $g_{i+1}$ have no common right divisor in $G^-$, therefore the right-maximal condition can also be written as:
\begin{equation}\label{eq:right_maximality_lattice_condition}
s^{-1}g_i^{-1} \vee g_{i+1} = e \textnormal{ for } 1 \leq i < k \tag{RM}
\end{equation}

In right $\ell$-groups with a strong order unit, right-normal factorizations exist and are unique:

\begin{thm} \label{thm:right_normal_facs_exist_and_unique}
Each element $g \in G^-$ has a unique right-normal factorization given by $g = g_kg_{k-1} \ldots g_1$, where $k = \min \{ i \geq 0: g \geq s^{-i} \}$ and
\begin{equation} \label{eq:right_normal_facs}
g_i = (g \vee s^{-i})(g \vee s^{-(i-1)})^{-1}.
\end{equation}
\end{thm}

\begin{proof}
Due to $s$ being a strong order unit, the integer $k$ exists. As $g \vee s^{-k} = g$, it is easily seen that with our choice of the factors $g_i$, we have
\[
g_j g_{j-1} \ldots g_1 = g \vee s^{-j}.
\]
In particular, $g_k g_{k-1} \ldots g_1 = g \vee s^{-k} = g$.

We now show that $g_i \in [s^{-1},e]$ ($1 \leq i \leq k$) which means that
\[
s^{-1} \leq (g \vee s^{-i})(g \vee s^{-(i-1)})^{-1} \leq e
\]
which by right-invariance is equivalent to
\[
s^{-1}(g \vee s^{-(i-1)}) \leq g \vee s^{-i} \leq g \vee s^{-(i-1)}.
\]
Using normality, the left term is seen to be $s^{-1}g \vee s^{-i}$; the first inequality now simply follows from the observation that $s^{-1}g \leq g$, whereas the second inequality follows from $s^{-i} \leq s^{-(i-1)}$.

We now show that the expression is indeed right-maximal. Assume otherwise, then we could find an index $1 \leq i < k$ and a decomposition $g_{i+1} = h^{\prime}h$ with $h^{\prime},h \in [s^{-1},e]$ and $h \neq e$ such that $hg_i \in [s^{-1},e]$. We could therefore get another factorization $g = h_kh_{k-1} \ldots h_1$ by definining
\[
h_j = \begin{cases}
g_j & i \not\in \{ i;i+1 \} \\
hg_i & j = i \\
h^{\prime} & j = i+1.
\end{cases}
\]
As all $h_j \geq s^{-1}$ ($1 \leq j \leq k$), a repeated application of \cref{lem:multipl_inequalities_with_normal_elements} would imply that $H := h_ih_{i-1} \ldots h_1 \geq s^{-i}$. But also $H \geq g$, which implied that $g \vee s^{-i} \leq H < g_i g_{i-1} \ldots g_1 = g \vee s^{-i}$, a contradiction!

Observe that right-maximality contains that $g_i = e$ for some index $1 \leq i \leq k$ implies $g_j = e$ for all indices $j > i$. When that occurs, $g = g_{i-1}g_{i-2} \ldots g_1$ which implies $g \geq s^{-(i-1)}$, also by \cref{lem:multipl_inequalities_with_normal_elements}. As that would contradict our choice of $k$, we deduce that all $g_i \neq e$ ($1 \leq i \leq k$).

Now it remains to prove that the right-normal factorization is indeed unique. Therefore, let $g = g_kg_{k-1} \ldots g_1 = h_l h_{l-1}\ldots h_1$ be two right-normal factorizations of $g$. It is sufficient to prove that $g_1 = h_1$, the rest will follow by induction.

So assume that $g_1 \neq h_1$. We may assume that $g_1 \not \leq h_1$. Then $g_1 \to h_1 \neq e$. Also, since $g_1 \to h_1 \geq g_1 \to s^{-1} = s^{-1}g_1^{-1}$ and $g_2 \vee s^{-1}g_1^{-1} = e$ (\cref{eq:right_maximality_lattice_condition}), we get $g_2 \not \leq g_1 \to h_1$. An inductive continuation of this argument, together with an application of \cref{eq:l_structure_left}, gives:
\[
(g_k g_{k-1} \ldots g_1) \to h_1 = g_k \to (g_{k-1} \to ( \ldots \to (g_1 \to h_1) )) \neq e,
\]
a contradiction to $g \leq h_1$.
\end{proof}

For an element $g \in G^-$, we call the quantity $k = \min \{ i \geq 0 : g \geq s^{-i} \}$ the \emph{length} of $g$ and denote it by $\lambda(g)$.

\begin{cor} \label{cor:all_dual_atoms_are_above_s'}
If $x \in X(G^-)$, then $\lambda(x) = 1$.
\end{cor}

\begin{proof}
Let $x = g_kg_{k-1} \ldots g_1$ be the right-normal factorization. As $\deg(x) = 1$, it is only possible that $\lambda(x) = k=1$.
\end{proof}

Below we will see that $s$ is also a strong order unit in the group $(G^{\mathrm{op}}, \lesssim)$, therefore one can also define the \emph{left-normal factorization} of an element $G^-$ as one that translates to a right-normal factorization in $(G^{\mathrm{op}}, \lesssim)$. We give a precise definition:

\begin{defi}
For an element $g \in G^-$, a \emph{left-normal factorization} is a sequence of elements $h_1,h_2, \ldots,h_k \in [s^{-1},e]$ such that:
\begin{enumerate}[1)]
\item $g = h_1 \ldots h_{k-1}h_k$,
\item $h_i \neq e$ for all $1 \leq i \leq k$, and
\item for all $1 \leq i < k$, there do \emph{not} exist $h,h^{\prime} \in G^-$ with $h \neq e$, $hh^{\prime} = h_{i+1}$ and $h_ih \in [s^{-1},e]$ (\emph{left-maximality}).
\end{enumerate}
\end{defi}

Note that left-maximality can also be expressed lattice-theoretically, in the spirit of \cref{eq:right_maximality_lattice_condition}. An expression $h_1 \ldots h_{k-1}h_k$ is therefore left-maximal if and only if for all $1 \leq i < k$,

\begin{equation} \label{eq:left_maximality_lattice_condition}
h_{i+1} \curlyvee h_i^{-1}s^{-1} = e. \tag{LM}
\end{equation}

This is also equivalent to $h_{i+1}^{-1} \wedge sh_i = e$ which - after left-multiplication by $s^{-1}$ - can be rewritten as

\begin{equation} \label{eq:left_maximality_lattice_condition_alt}
h_i \wedge s^{-1}h_{i+1}^{-1} = s^{-1}. \tag{LM'}
\end{equation}

In order to guarantee that this definition is indeed meaningful, one must prove that $s$ is a strong order unit with regard to the opposite order $\lesssim$ and that the respective strong order intervals coincide as well.

\begin{lem} \label{lem:normal_element_less_and_lesssim}
Let $t \in G$ be a normal element, then $t \leq g \Leftrightarrow t \lesssim g$.
\end{lem}

\begin{proof}
Since $t$ is normal, $t \leq g \Leftrightarrow g^{-1} \leq t^{-1} \Leftrightarrow t \lesssim g$.
\end{proof}

\begin{prop} \label{prop:strong_order_intervals_coincide}
If $s$ is a strong order unit, then for all $k \geq 0$, the strong order intervals with respect to $\leq$ and $\lesssim$ coincide, i.e.
\[
[s^{-k},e]_{\leq} = [s^{-k},e]_{\lesssim}.
\]
\end{prop}

\begin{proof}
$s^{-k} \leq g \Leftrightarrow s^{-k} \lesssim g$ holds by the lemma above, as $s^{-k}$ is a normal element. By right-invariance, we also get $g \leq e \Leftrightarrow e \leq g^{-1} \Leftrightarrow g \lesssim e$.
\end{proof}

From \cref{thm:right_normal_facs_exist_and_unique}, we can now easily deduce:

\begin{thm} \label{thm:left_normal_facs_exist_and_unique}
Let $g \in G^-$, then $g$ has a unique left-maximal factorization $g = h_1 \ldots h_{k-1}h_k$ with $k = \lambda(g)$, whose factors are given by
\begin{equation} \label{eq:left_normal_facs}
h_i = (g \curlyvee s^{-(i-1)})^{-1}(g \curlyvee s^{-i}).
\end{equation}
\end{thm}

\begin{proof}
By \cref{prop:strong_order_intervals_coincide}, the factors $h_i$ are indeed in the interval $[s^{-1},e]$. As a left-normal factorization is just a right-normal factorization in the opposite group, \cref{eq:left_normal_facs} simply follows from applying \cref{eq:right_normal_facs} to the right $\ell$-group $(G^{\mathrm{op}},\lesssim)$. As \cref{prop:strong_order_intervals_coincide} also implies that $g \geq s^{-k} \Leftrightarrow g \gtrsim s^{-k}$, the factorizations must be of the same length.
\end{proof}

Introducing chain conditions, there also exists a factorization into dual atoms.

Call a right $\ell$-group $G$ \emph{noetherian}, if:
\begin{enumerate}[1)]
\item For each $g \in G$, the sublattice $g^{\uparrow}$ fulfills the descending chain condition,
\item For each $g \in G$, the sublattice $g^{\downarrow}$ fulfills the ascending chain condition
\end{enumerate}
We remark that these two conditions are \emph{not} equivalent; in general, one condition can only be deduced from the other under left- \emph{and} right-invariance.

Recall that $X(G^-) := \{ x \in G: x \prec  e \}$. The proof of the next proposition is quite standard:

\begin{prop} \label{prop:irred_facts_exist}
If $G$ is a noetherian right $\ell$-group and $g \in G^-$, then there is a finite sequence $x_1, \ldots, x_k$ in $X(G^-)$ for some $k \geq 0$ such that $g = x_k x_{k-1} \ldots x_1$.
\end{prop}

\begin{proof}
Easy.
\end{proof}

Note that a product with $k = 0$ is an empty product and therefore equal to $e$.

A factorization of this type will be denoted as an \emph{atomic factorization} and call the integer $k$ its \emph{length}. As with left- or right-normal factorizations, we will simply refer to an expression of the type $(g=)x_kx_{k-1}\ldots x_1$ as an atomic factorization (of $g$) if the factors are dual atoms.

In a \emph{modular} noetherian right $\ell$-group, we have the following theorem:

\begin{prop} \label{prop:irred_facts_have_same_length}
If $G$ is a modular noetherian right $\ell$-group and $g \in G^-$. For two atomic factorizations $g = x_kx_{k-1}\ldots x_1 = y_ly_{l-1}\ldots y_1$, there holds the equality $k=l$.
\end{prop}

\begin{proof}
For $0 \leq i \leq k$, set $a_i = x_i x_{i-1} \ldots x_1$ and, similarly, set $b_j^{\prime} := y_jy_{j-1}\ldots y_1$ for $0 \leq j \leq l$.

Then both $g = a_k \prec a_{k-1} \prec \ldots \prec a_1 \prec a_0 = e $ and $g = b_l \prec b_{l-1} \prec \ldots \prec b_1 \prec b_0 = e$ are maximal chains in the modular lattice $[g,e]$. Therefore, by the Jordan-Hölder-theorem for modular lattices (\cite[Theorem 374]{Graetzer-Lattice}), their length is the same.
\end{proof}

\begin{prop}
The function $\deg^{\prime} : G^- \to (\Z_{\geq 0},+)$, defined by $\deg^{\prime}(g) = k$ whenever there is an atomic factorization $g = x_kx_{k-1}\ldots x_1$, is a homomorphism of monoids.
\end{prop}

\begin{proof}
By \cref{prop:irred_facts_exist} and \cref{prop:irred_facts_have_same_length}, the map $\deg^{\prime}$ is well-defined. It is clear that $\deg^{\prime}(e) = 0$. Let $g,h \in G^-$. If $\deg^{\prime}(g) = k$ and $\deg^{\prime}(h) = l$, then there are atomic factorizations $g = x_kx_{k-1} \ldots x_1$ and $h = y_l y_{l_1} \ldots y_1$, so $gh = x_kx_{k-1} \ldots x_1y_l y_{l_1} \ldots y_1$ is an atomic factorization with $k+l$ factors. Therefore, $\deg^{\prime}(gh) = k+l = \deg^{\prime}(g) + \deg^{\prime}(h)$. This proves that $\deg^{\prime}$ is a monoid homomorphism.
\end{proof}

\begin{prop} \label{prop:degree_homomorphism_exists}
There is a unique group homomorphism $\deg: G \to \Z$ that extends $\deg^{\prime}$ in the sense that $\deg(g) = \deg^{\prime}(g)$ for all $g \in G^-$. This group homomorphism is defined by $\deg(g_2^{-1}g_1) := \deg^{\prime}(g_1) - \deg^{\prime}(g_2)$ for all $g_1,g_2 \in G^-$.
\end{prop}

\begin{proof}
Enlarging the codomain, we can consider $\deg^{\prime}$ as a monoid homomorphism $G^-$ to $\Z$. By \cref{prop:extending_homomorphisms}, there is a unique extension to a group homomorphism $\deg: G \to \Z$ that is of the stated form.
\end{proof}

\begin{rema}
Note that the existence of a degree homomorphism $\deg : G \to \Z$ with $\deg(x) = 1$ for all $x \in X(G^-)$ does \emph{not} imply the modularity of $G$. The braid groups $B_n$ also admit such a degree homomorphism despite being far from modular, simply because they have a length-balanced presentation.
\end{rema}

Modularity easily implies the following useful \emph{parallelogram identity} for the degree homomorphism:

\begin{lem} \label{lem:parallelogram_identity}
For all $g,h \in G$, we have the equality
\[
\deg(g) + \deg(h) = \deg(g \vee h) + \deg(g \wedge h).
\]
\end{lem}

\begin{proof}
By the diamond lemma (\cref{lem:diamond_lemma}), we have an isomorphism $[g \wedge h, h] \cong [g, g \vee h]$, due to which $\len([g \wedge h, h]) = \len([g, g \vee h])$. Therefore, 
\[
\len([g \wedge h, h]) = \len([(g \wedge h)h^{-1},e]) = \deg((g \wedge h)h^{-1}) = \deg(g \wedge h) - \deg(h).
\]
Similarly, $\len([g, g \vee h]) = \deg(g)- \deg(g \vee h)$. Therefore $\deg(g)- \deg(g \vee h) = \deg(g \wedge h) - \deg(h)$, which proves the equation.
\end{proof}

From now on, we will always implicitly assume that $G$ is a modular, noetherian right $\ell$-group with strong order unit $s^{-1}$ and respective degree function denoted by $\deg$.

We now define a very important sequence, the \emph{index sequence} of an element $g \in G^-$:

\begin{defi}
For an element $g \in G^-$, we define its \emph{index sequence} as
\begin{align*}
\iota_{\ast}(g) : \Z_{> 0} & \to \Z_{\geq 0} ; \\
i & \mapsto \deg(g \vee s^{-i}) - \deg(g \vee s^{-(i-1)}).
\end{align*}
\end{defi}

\begin{rema}
Recall that by \cref{thm:right_normal_facs_exist_and_unique} the right-normal factors of $g$ are given by 
\[
g_i = (g \vee s^{-i})(g \vee s^{-(i-1)})^{-1} \quad (1 \leq i \leq \lambda(g)).
\]
As $g \vee s^{-i} = g$ for $i \leq \lambda(g)$, we get that
\[
\iota_i(g) = \begin{cases}
\deg(g_i) & i \leq \lambda(g) \\
0 & i > \lambda(g).
\end{cases}
\] 
\end{rema}

An extremely useful property of the sequence $\iota_i(g)$ is the following:

\begin{thm} \label{thm:iota_is_nonincreasing}
For any $g \in G^-$, the sequence $\iota_i(g)$ is non-increasing.
\end{thm}

\begin{proof}
Let $g = g_kg_{k-1} \ldots g_1$ be the right-normal factorization of $g$. By the preceding remark, it is sufficient to show that $\deg(g_{i+1}) \leq \deg(g_i)$ for $1 \leq i < k$. By \cref{eq:right_maximality_lattice_condition}
we have
\[
g_{i+1} \vee s^{-1}g_i^{-1} = e
\]
We now calculate
\begin{align*}
\deg(s^{-1}) & \geq \deg(g_{i+1} \wedge s^{-1}g_i^{-1}) \\
& = \deg(g_{i+1}) + \deg(s^{-1}g_i^{-1}) - \deg(g_{i+1} \vee s^{-1}g_i^{-1}) & \textnormal{(\cref{lem:parallelogram_identity})} \\
& = \deg(g_{i+1}) + \deg(s^{-1}) - \deg(g_i),
\end{align*}
from which it follows that $\deg(g_i) \geq \deg(g_{i+1})$.
\end{proof}

We now need to define a special class of elements in $G^-$ that is particularly useful for lattice-theoretic considerations in $G^-$:

\begin{defi}
Let $g \in G^-$ have the right-normal factorization $g = g_k g_{k-1} \ldots g_1$. Then $g$ is called \emph{$d$-homogeneous} for an integer $d \geq 1$ if $\deg(g_i) = d$ for all $1 \leq i \leq k$.
\end{defi}

We now show a remarkable symmetry property that holds in modular right $\ell$-groups. Not only do right- and left-normal factorizations have the same number of factors, their factors even have the same size!

\begin{prop} \label{prop:left_right_symmetry_of_degrees}
Let $g \in G^-$ and let $g = g_kg_{k-1} \ldots g_1$ and $g = h_1 \ldots h_{k-1}h_k$ be the right-normal resp. the left-normal factorization. Then $\deg(h_i) = \deg(g_i)$ ($1 \leq i \leq k$).
\end{prop}

\begin{proof}
By the expressions provided by \cref{thm:right_normal_facs_exist_and_unique} and \cref{thm:left_normal_facs_exist_and_unique}, it is sufficient to show that $\deg(g \vee s^{-i}) = \deg(g \curlyvee s^{-i})$ for all $i$. We calculate
\begin{align*}
\deg(g \curlyvee s^{-i}) & = \deg( (g^{-1} \wedge s^i)^{-1} ) \\
& = - \deg(g^{-1} \wedge s^i) \\
& = \deg(g^{-1} \vee s^i) - \deg(g^{-1}) - \deg(s^i) & \textnormal{(\cref{lem:parallelogram_identity})} \\
& = \deg(s^i) + \deg(s^{-i} \vee g) + \deg(g^{-1}) - \deg(g^{-1}) - \deg(s^i) \\
& = \deg(s^{-i} \vee g).
\end{align*}
\end{proof}

Despite being only in need of the case $d=1$, we state and prove the following propositions in their full generality:

\begin{prop} \label{prop:homogeneous_rightnormal_implies_leftnormal}
Let $g \in G^-$ be $d$-homogeneous with the right-normal factorization $g = g_kg_{k-1} \ldots g_1$. Then this factorization is left-normal as well.
\end{prop}

\begin{proof}
By \cref{eq:right_maximality_lattice_condition}, we have $g_{i+1} \vee s^{-1}g_i^{-1} = e$ ($1 \leq i < k$). We can can now calculate
\begin{align*}
\deg(g_{i+1} \wedge s^{-1}g_i^{-1}) & = \deg(g_{i+1}) + \deg(s^{-1}g_i^{-1}) - \deg(g_{i+1} \vee s^{-1}g_i^{-1}) \textnormal{(\cref{lem:parallelogram_identity})} \\
& = \deg(g_{i+1}) + \deg(s^{-1}) - \deg(g_i^{-1}) - 0 \\
& =  \deg(s^{-1}).
\end{align*}
As $g_{i+1} \wedge s^{-1}g_i^{-1} \geq s^{-1}$, this implies $g_{i+1} \wedge s^{-1}g_i^{-1} = s^{-1}$, i.e. the factorization is left-maximal by \cref{eq:left_maximality_lattice_condition_alt} (Watch out here! The $g_i$'s are indexed differently from the factors in the equation!).
\end{proof}

The following analysis builds on a result of Rump.

\begin{thm} \label{thm:strong_order_interval_is_geometric}
Let $G$ be a modular noetherian right $\ell$-group. If the element
\[
s^{-1} := \bigwedge X(G^-)
\]
exists, then $s$ is a strong order unit and the strong order interval $[s^{-1},e]$ is a modular geometric lattice.
\end{thm}

\begin{proof}
These statements are \cite[Propositions 14+15.]{Rump-Geometric-Garside}.
\end{proof}

We later need to make use of the following fact:

\begin{prop} \label{prop:join_of_atoms_is_strong_order_unit}
Let $s$ be as in \cref{thm:strong_order_interval_is_geometric}, then
\[
s = \bigvee \{ g \in G : g \succ e \}
\]
and the interval $[e,s]$ is a modular geometric lattice.
\end{prop}

\begin{proof}
As soon as we has proven the first statement, the second will be clear. By right-invariance, the first statement is equivalent to $e = \bigvee \{ g \in G : g \succ s^{-1} \}$.

Let $\delta$ be the dimension of $[s^{-1},e]$ and let  $Y = \{y_1,\ldots,y_{\delta} \}$ be a basis of $[s^{-1},e]$. As $[s^{-1},e]$ is geometric, $Y$ consists of atoms in $[s^{-1},e]$. It is clear that $e = \bigvee Y \leq \bigvee \{ g \in G : g \succ s^{-1} \}$. If we assumed that $>$ holds, this would imply the existence of a $z \succ s^{-1}$ such that $h := z \vee \bigvee Y > e$ and $Y \cup \{ z \}$ was a basis of $[s^{-1},h]$ consisting of $\delta + 1$ atoms in this lattice.

This would imply by \cref{lem:reflection_of_frames} that $[s^{-1}h^{-1},e] \cong [s^{-1},h]$ had a basis consisting of $\delta + 1$ dual atoms of $G^-$, so
\[
s^{-1} = \bigwedge X(G^-) \leq s^{-1}h^{-1} < s^{-1},
\]
a contradiction! Therefore, we have equality, i.e. $e = \bigvee \{ g \in G : g \succ s^{-1} \}$.
\end{proof}

We can now understand the structure of meet-irreducibles better.

\begin{prop} \label{prop:meet_irreducibles_are_1_homogeneous}
Assume that $G$ is a modular noetherian right $\ell$-group that contains the strong order unit $s$ defined by $s^{-1} := \bigwedge X(G^-)$. An element $g \in G^-$ is meet-irreducible in $G^-$ if and only if it is $1$-homogeneous with respect to $s$.
\end{prop}

\begin{proof}
Assume that $g$ is meet-irreducible in $G^-$ and let $g = g_kg_{k-1}\ldots g_1$ be the right-normal and $g = h_1 \ldots h_{k-1}h_k$ be the left-normal factorization. By \cref{thm:iota_is_nonincreasing}, it is sufficient to prove $\deg(g_1) = 1$. Assuming otherwise, we would have $\deg(h_1) = \deg(g_1) > 1$, by \cref{prop:left_right_symmetry_of_degrees}. But then we could represent $h_1 = h \wedge h^{\prime}$ with $h_1 \neq h, h^{\prime}$, due to $[s^{-1},e]$ being dually atomistic. However, we would then get a representation $g = h h_2 \ldots h_{k-1}h_k \wedge h^{\prime} h_2 \ldots h_{k-1}h_k$ with both terms unequal to $g$.

In order to show the other direction, assume that the element $g \in G^-$ is \emph{not} meet-irreducible in $G^-$. Then there are elements $y,y^{\prime} \in G^-$ with $g \prec y,y^{\prime}$ such that $g = y \wedge y^{\prime}$. Let furthermore be $\gamma := y \vee y^{\prime}$. Then $\len([g,y \vee y^{\prime}]) = 2$ which implies that the element $\tilde{h} := g \gamma^{-1}$ has $\deg(\tilde{h}) = 2$.

But then $\tilde{h} \gtrsim g$. Also, as $\tilde{h} = y \gamma^{-1} \wedge y^{\prime} \gamma^{-1}$ is a meet of dual atoms, we get that $\tilde{h} \in [s^{-1},e]$. Therefore, if $g = h_1 \ldots h_{k-1}h_k$ is the left-normal factorization, we have $\tilde{h} \gtrsim h_1$. In particular, $\deg(h_1) \geq \deg(\tilde{h}) > 1$. Due to \cref{prop:left_right_symmetry_of_degrees} we see that $\iota_1(g) = \deg(h_1) > 1$, i.e. $g$ is not $1$-homogeneous.
\end{proof}

If $L$ is a lattice that is bounded from above, recall that an element $x \in L$ a \emph{dual chain}, whenever $[x,1_L]$ is a chain.

\begin{cor} \label{prop:meet_irreducibles_are chains}
An element $g \in G^-$ is meet-irreducible in $G^-$ if and only if $g$ is a dual chain.
\end{cor}

\begin{proof}
It is trivial that a dual chain is meet-irreducible. On the other hand, let $g \in G^-$ be meet-irreducible and let $g = g_kg_{k-1} \ldots g_1$ be the right-normal factorization, with all $\deg(g_i) = 1$, by \cref{prop:meet_irreducibles_are_1_homogeneous}. Let $h \in G^-$ be such that $h \geq g$ and set $l := \lambda(h) \leq k$. We claim that $h = g_lg_{l-1} \ldots g_1$. It is clear, that $\deg(h) \geq l$. On the other hand, $h = h \vee s^{-l} \geq g \vee s^{-l} = g_l g_{l-1} \ldots g_1$. As $\deg(g_l g_{l-1} \ldots g_1) = l$, this confirms the claimed factorization. Therefore, $[g,e] = \{ g_l g_{l-1} \ldots g_1 : 0 \leq l \leq k \}$ is a chain; therefore, $g$ is a dual chain.
\end{proof}

\section{The center of the strong order interval} \label{sec:center_of_strong_order_interval}

We want to construct a lattice decomposition for modular noetherian right $\ell$-groups with strong order unit. We will start with a lattice decomposition of a suitably chosen strong order interval and extend this to the entirety of $G$. In order to do that, we need the classic theorem that modular geometric lattices can uniquely be written as a product of directly irreducible modular geometric lattices.

A lattice $L$ is called \emph{directly irreducible} if in an internal product $L = L_1 \times L_2$, either $|L_1| = 1$ or $|L_2| = 1$.

\begin{thm} \label{thm:decomposition_geometric_lattices}
Let $L$ be a modular geometric lattice. Then there is an internal decomposition $L = \prod_{i=1}^k L_i$, where each $L_i$ is a directly irreducible modular geometric lattice. Furthermore, this decomposition is unique up to a permutation of the factors.
\end{thm}

\begin{proof}
See \cite[Theorem 393]{Graetzer-Lattice} for the existence of the decomposition and \cite[Corollary 277]{Graetzer-Lattice} for the uniqueness.
\end{proof}

\begin{rema}
A cautious reader may have noted that the isomorphisms in \cite[Corollary 277]{Graetzer-Lattice} are \emph{external} product decompositions. This means that the isomorphism classes of indecomposable factors are determined but not a priori their position in an internal decomposition. However, \cite[Theorem 276]{Graetzer-Lattice} and the preceding discussion make clear that two sets of indecomposable internal factors would admit a common refinement - which is only possible if the decompositions are already equal, up to permutation.

We will sometimes refer to this corollary and we will always mean its implications on \emph{internal} decompositions! 
\end{rema}

From the decomposition of a modular geometric lattice into directly irreducible lattices and the classification of irreducible modular geometric lattices (\cite[Chapter V., Section 5.]{Graetzer-Lattice}), whose duals are still modular geometric, it is clear that the notions \emph{geometric} and \emph{dually geometric} are actually equivalent for bounded modular lattices. As all geometric lattices considered in this work are bounded and modular, we will for simplicity speak about geometricity while making use of the dual geometricity of these lattices.

From now on, we will always assume that $G$ is a modular noetherian right $\ell$-group where the strong order unit $s$ with $s^{-1} := \bigwedge X(G^-)$ exists, so \cref{thm:strong_order_interval_is_geometric} applies.

For in this section we will work a lot with product decompositions of $[s^{-1},e]$, it makes sense to look at \emph{central} elements in the strong order interval:

\begin{defi}[{\cite[I, Definition 3.2.]{maeda}}]
Let $L$ be a bounded lattice. An element $z \in L$ is called \emph{central} if and only if there is an isomorphism $L \overset{\sim}{\longrightarrow} L_1 \times L_2$, with $L_1,L_2$ being bounded lattices, such that $z$ is sent to the element $(0_{L_1},1_{L_2})$. 
\end{defi}

We will call the totality of central elements in $L$ the \emph{center} of $L$ and denote it by $\Cent(L)$.

The following theorem will be central for our further investigations of $\Cent(L)$:

\begin{thm} \label{thm:cent_is_boolean}
If $L$ is a bounded lattice, then $\Cent(L)$ is a Boolean sublattice of $L$.
\end{thm}

\begin{proof}
\cite[I, Satz 3.4.]{maeda}.
\end{proof}

The following statements can be seen easily:

\begin{prop} \label{prop:irreducible_means_trivial_center}
Let $L$ be a modular geometric lattice. Then $L$ is directly irreducible if and only if $\Cent(L) = \{ 1_L, 0_L \}$.
\end{prop}

\begin{proof}
This follows directly from the definitions.
\end{proof}

\begin{prop} \label{prop:dual_atoms_of_center_are_irred_factors}
Let $L$ be a modular geometric lattice with an internal factorization $L \cong L_1 \times \ldots \times L_k$, where all factors $L_i$ are non-trivial and directly irreducible. Then $|X(\Cent(L))| = k$.
\end{prop}

\begin{proof}
For sake of simplicity, we identify $L$ with the external product $L_1 \times \ldots \times L_k$. It is clear that the elements $z_i = (\epsilon_j^i)_{j=1}^k$ are central in $L_1 \times \ldots \times L_k$ where
\[
\epsilon_j^i = \begin{cases}
0_{L_i} & j = i \\
1_{L_j} & j \neq j.
\end{cases}
\]
It is immediate that $[z_i,1_L] \cong L_i$. If there was an element $\tilde{z} \in L$ with $z_i < \tilde{z} < 1_L$ for some $i$, then $\tilde{z}$ would be a nontrival central element in $[z_i,1_L]$, contradicting \cref{prop:irreducible_means_trivial_center}. Also, there are no further elements in $X(\Cent(L))$, as $\bigwedge_{i=1}^k z_i = 0_L$.
\end{proof}

Let the sublattice $\Ccal(G^-) \subseteq G^-$ be defined as $\Ccal(G^-) := \Cent ([s^{-1},e])$. Obviously, $e$ and $s^{-1}$ are elements of $\Ccal(G^-)$; in particular, they are the top resp. bottom element thereof.

The aim of this section is to show that $\Ccal(G^-)$ is closed under the $\to$-operation defined in \cref{subs:right_l_groups}.

In order to achieve this, let $\{ z_1, \ldots, z_k \} = X(\Ccal(G^-))$ be the dual atoms in $\Ccal(G^-)$; note that these need not be dual atoms in $[s^{-1},e]$ resp. $G^-$. We now have a canonical lattice isomorphism
\[
\varphi: \prod_{i=1}^k [z_i,e] \longiso [ s^{-1},e ] \quad ; \quad
(a_i)_{i=1}^k \longmapsto \bigwedge_{i=1}^k a_i
\]
whose inverse map is given by
\[
\varphi^{-1}(a) = (a \vee z_i)_{i=1}^k.
\]
The intervals $[z_i,e]$ are irreducible modular geometric lattices, and so are all of their intervals. By means of $\varphi$, for all $a \in [s^{-1},e]$ we have (canonical) isomorphisms
\begin{align}
[a,e] & \cong \prod_{i=1}^k [a_i,e] \label{eq:upper_set_decomposition}\\
[s^{-1},a] & \cong \prod_{i=1}^k [z_i,a_i] \label{eq:lower_set_decomposition}
\end{align}
where $(a_i)_{i=1}^k = \varphi^{-1}(a)$.

We remark here that each upper and lower set therefore decomposes as a product of irreducible modular geometric lattices, depending on how many components of $\varphi^{-1}(a)$ are equal to $e$ resp. $z_i$. This motivates us to introduce the following quantities:
\begin{align*}
U(a) & :=  \# \{ i \in [1,k] : a \vee z_i  = e \} \\
L(a) & := \# \{ i \in [1,k] : a \vee z_i  = z_i \}. \\
M(a) & := k - (U(a) + L(a)).
\end{align*}
Here, $k := |X(\Ccal(G^-))|$. Obviously, these three quantities are nonnegative and sum up to $k$.

By \cref{eq:upper_set_decomposition} it becomes clear that the number of irreducible factors of $[a,e]$ is $L(a) + M(a)$, whereas \cref{eq:lower_set_decomposition} implies that the number of irreducible factors of $[s^{-1},a]$ is $U(a) + M(a)$.

For the rest of this section, fix $k$ as the number of dual atoms in $\Ccal(G^-)$.

\begin{lem} \label{lem:central_eq_to_Ma=0}
An element $a \in [s^{-1},e]$ is in $\Ccal(G^-)$ if and only if $M(a) = 0$.
\end{lem}

\begin{proof}
After decomposing $[s^{-1},e]$ into its irreducible components, this follows from \cref{prop:irreducible_means_trivial_center}.
\end{proof}

We can now prove the following lemma:

\begin{lem} \label{lem:equality_for_Ua_and_La}
For any $a \in [s^{-1},e]$, we have the equality
\begin{equation} \label{eq:equality_Ua_La}
U(s^{-1}a^{-1}) = L(a).
\end{equation}
\end{lem}

\begin{proof}
Let $\varphi^{-1}(a) = (a_i)_{i=1}^k$. In the decomposition $[s^{-1},a] \cong \prod_{i=1}^k [z_i,a_i]$, exactly $L(a)$ factors are trivial. Right-multiplication by $a^{-1}$ gives an isomorphism $[s^{-1},a] \longiso [s^{-1}a^{-1},e]$. Therefore, letting $\varphi^{-1}(s^{-1}a^{-1}) = (b_i)_{i=1}^k$, we must have the same number of trivial factors in the decomposition $[s^{-1}a^{-1},e] \cong \prod_{i=1}^k [b_i,e]$, which is $U(s^{-1}a^{-1})$.
\end{proof}

\begin{lem} \label{lem:ba'_in_cent}
Let $a,b \in \Ccal(G^-)$ be such that $b \leq a$, then $ba^{-1} \in \Ccal(G^-)$ as well. In this case, $L(ba^{-1}) = L(b) - L(a)$.
\end{lem}

\begin{proof}
Let $a \in \Ccal(G^-)$. We first prove the special case that $b = s^{-1}$. By \cref{lem:central_eq_to_Ma=0}, $M(a) = 0$ which implies $L(a) + U(a) = k$.

Using \cref{eq:equality_Ua_La} twice, we get:
\begin{align*}
L(a) & = U(s^{-1}a^{-1}) \\
& = k - L(s^{-1}a^{-1}) - M(s^{-1}a^{-1}) \\
& = k - U(s^{-1}as) - M(s^{-1}a^{-1}) \\
& = k - U(a) - M(s^{-1}a^{-1}) \\
& = L(a) - M(s^{-1}a^{-1}),
\end{align*}
where the equality $U(s^{-1}as) = U(a)$ follows from the fact that $a \mapsto s^{-1}as$ restricts to a lattice automorphism of $[s^{-1},e]$ (\cref{lem:conjugation_by_s_is_automorphism}). Thus, we have shown that $M(s^{-1}a^{-1}) = 0$. So, by \cref{lem:central_eq_to_Ma=0}, we have shown that $s^{-1}a^{-1} \in \Ccal(G^-)$.

We now prove the statement for general $a,b$ with $b  \leq a$. By restriction, it is clear that $b$ is a central element of the lattice $[s^{-1},a]$. As right-multiplication by $a^{-1}$ induces an isomorphism $[s^{-1},a] \cong [s^{-1}a^{-1},e]$, we see that $ba^{-1}$ - the image of $b$ - is a central element in the lattice $[s^{-1}a^{-1},e]$. As $s^{-1}a^{-1}$ is central in $[s^{-1},e]$ and $ba^{-1}$ is central in $[s^{-1}a^{-1},e]$, we infer that $ba^{-1}$ is central in $[s^{-1},e]$ as well.

The last part of the lemma easily follows from the fact that the interval $[ba^{-1},e] \cong [b,a]$ decomposes as a direct product of $L(b)-L(a)$ directly indecomposable factors.
\end{proof}

\begin{thm} \label{thm:center_closed_under_arrow}
$\Ccal(G^-)$ is closed under the $\to$-operation defined in \cref{subs:right_l_groups}.
\end{thm}

\begin{proof}
Let $a,b \in \Ccal(G^-)$. As $a \to b = (a \wedge b)a^{-1}$ and $a \wedge b \leq a$, it follows from \cref{lem:ba'_in_cent} that $a \to b \in \Ccal(G^-)$.
\end{proof}

As $\Ccal(G^-)$ is a Boolean algebra (\cref{thm:cent_is_boolean}), we can define the operator 
\begin{align*}
D: \Ccal(G^-) & \to \Ccal(G^-) \\
a & \mapsto (s^{-1}a^{-1})^{\prime},
\end{align*}
where $x^{\prime}$ denotes the complement of $x$ in the Boolean algebra $\Ccal(G^-)$.

Note that $U(D(a)) = U((s^{-1}a^{-1})^{\prime}) = k - U(s^{-1}a^{-1}) = k - (k - U(a)) = U(a)$ for all $a \in \Ccal(G^-)$. Also note that since $a \mapsto s^{-1}a^{-1}$ and $a \mapsto a^{\prime}$ are bijections from $\Ccal(G^-)$ to itself, so is the map $D$.

One useful property of the map $D$ is that it is capable of producing right-maximal expressions:

\begin{prop} \label{prop:duality_gives_right_maximal_expression}
If Let $x \in \Ccal(G^-)$, then $y := D(x)$ is the unique element $y \in [s^{-1},e]$ such that $\deg(y) = \deg(x)$ and the expression $yx$ is right-maximal. Furthermore, $z \vee D(x) = e$ implies that $zx \geq s^{-1}$.
\end{prop}

\begin{proof}
In order for the expression $yx$ to be right-maximal it is necessary to have $y \vee s^{-1}x^{-1} = e$. As $\deg(s^{-1}x^{-1}) = \deg(s^{-1}) - \deg(x)$, the condition $\deg(y) = \deg(x)$ forces $y \wedge s^{-1}x^{-1} = s^{-1}$, i.e. $y$ is a complement of $s^{-1}x^{-1}$ in $[s^{-1},e]$. As $s^{-1}x^{-1}$ is central (\cref{lem:ba'_in_cent}), there is only one complement, which is $(s^{-1}x^{-1})^{\prime}$.

If $z \vee D(x) = e$, then $z \geq s^{-1}x^{-1}$, which implies $zx \geq s^{-1}$
\end{proof}

It is notable that $D$ preserves important properties of the elements in $\Ccal(G^-)$:
\begin{prop} \label{prop:D_preserves_deg_and_U}
Let $z \in \Ccal(G^-)$, then
\begin{align*}
\deg(D(z)) & = \deg(z) ; \\
U(D(z)) & = U(z).
\end{align*}
\end{prop}

\begin{proof}
The first equality has already been shown in the preceding proof.

For the second equality, we might as well prove that $L(D(z)) = L(z)$, as $M(z) = M(D(z))=0$ (\cref{lem:central_eq_to_Ma=0}). Using \cref{lem:equality_for_Ua_and_La}, we calculate
\[
L(z) = U(s^{-1}z^{-1}) = L((s^{-1}z^{-1})^{\prime}) = L(D(z)).
\]
\end{proof}

We will also make use of the following property:

\begin{lem} \label{lem:Dz_contains_atoms_not_in_image_of_z*}
Let $z \in \Ccal(G^-)$, then for $g \in X(G^-)$ we have the equivalence:
\[
g \geq D(z) \Leftrightarrow \not \exists h \in [s^{-1},e] : z \to h = g.
\]
\end{lem}

\begin{proof}
The image of $[s^{-1},e] \to [s^{-1},e]$; $h \mapsto (z \to h)$ is $[s^{-1}z^{-1},e]$. As $s^{-1}z^{-1}$ and $D(z)$ are complements to each other and central, an element of $X(G^-)$ must either be in $[s^{-1}z^{-1},e]$ and therefore is an image under $h \mapsto (z \to h)$, or be in $[D(z),e]$ and there lies over $D(z)$.
\end{proof}

We repeat some basic facts on \emph{$L$-algebras}. For more on Rump's theory of $L$-algebras, see the foundational article \cite{self_similarity},

\begin{defi}
An \emph{$L$-algebra} is a triple $(X,\to,e)$ where $e$ is a distinguished element of $e$ and $\to : X \times X \to X$ is a binary operation, such that the following identities are fulfilled for any $x,y,z \in X$:
\begin{align*}
x \to x & = x \to e = e, \\
e \to x & = x, \\
(x \to y) \to (x \to z) & = (y \to x) \to (y \to z) \\
x \to y = y \to x = e & \Leftrightarrow x = y.
\end{align*}
\end{defi}

It can be shown that each $L$-algebra carries a canonical partial order $\leq$ defined by $x \leq y : \Leftrightarrow x \to y = e$ (\cite[Proposition 2.]{self_similarity}).  From the definition it is clear that under this order, $x \leq e$ for all $x \in X$.

The following fact follows from the identities of \cref{lem:equations_for_arrow}:

\begin{prop} \label{prop:cones_are_l_algebras}
If $G$ is a right $\ell$-group with negative cone $G^-$, then $G^-$ becomes an $L$-algebra under the arrow operation $x \to y = yx^{-1} \wedge e$.
\end{prop}

If there are no nontrivial order relations on $X$, i.e. if the restriction of $\leq$ to $X \setminus \{ e \}$ is the trivial order, then $X$ is called a \emph{discrete} $L$-algebra.

\begin{defi}[\cite{Rump-Geometric-Garside}]
Let $X$ be a discrete $L$-algebra. Then a bijection $D: X \setminus \{ e \} \to X \setminus \{ e \}$ is called a \emph{duality} if for all $x,y \in X \setminus \{ e \}$ with $x \neq y$ we have
\[
D(x \to y) = (y \to x) \to D(y).
\]
\end{defi}

\begin{rema} \label{rema:l_algebras_and_cycle_sets}
By results of Rump \cite{rump_duality}, we now know that the concept of duality is not only meaningful for discrete $L$-algebras but also for a broader class of $L$-algebras.

Discrete $L$-algebras with duality, however, are in $1-1$-correspondence with non-degenerate cycle sets, algebraic structures that parametrize non-degenerate set-theoretic solutions to the Yang-Baxter equation \cite[Theorem 2.]{Rump-Geometric-Garside}.
\end{rema}

\begin{thm} \label{thm:atoms_of_center_are_l_algebra_with_duality}
$\tilde{X}(\Ccal(G^-))$ is an $L$-algebra with duality.
\end{thm}

\begin{proof}
We first show that $\tilde{X}(\Ccal(G^-))$ is an $L$-algebra at all. By \cref{prop:cones_are_l_algebras} and \cref{thm:center_closed_under_arrow}, it is sufficient to prove that for $a,b \in X(\Ccal(G^-))$ with $a \neq b$, also $a \to b \in X(\Ccal(G^-))$.

In this case, $L(a) = L(b) = 1$ and $L(a \wedge b) = 2$. By the last part of \cref{lem:ba'_in_cent}, we get 
\[
L(a \to b) = L((a \wedge b)a^{-1}) = L(a \wedge b) - L(a) = 2-1 = 1,
\]
implying that $a \to b \in X(\Ccal(G^-))$.

We now show that $X(\Ccal(G^-)) \to X(\Ccal(G^-)); \ a \mapsto D(a) := (s^{-1}a^{-1})^{\prime}$ restricts to a duality of $\tilde{X}(\Ccal(G^-))$. Note that this map is well-defined due to the second equality in \cref{prop:D_preserves_deg_and_U}. Now take $x,y \in X(\Ccal(G^-))$ with $x \neq y$. Using \cref{lem:Dz_contains_atoms_not_in_image_of_z*}, we see that $D(x \wedge y)$ is the meet of all elements in $X(\Ccal(G^-))$ that are not in the image of $z \mapsto (x \wedge y) \to z$.

By $(x \wedge y) \to z = (x \to y) \to (x \to z)$ (\cref{eq:s3}), we see that
\[
D(x \wedge y) = D(x \to y) \wedge \left( (x \to y) \to D(x) \right).
\]
On the other hand, expressing $(x \wedge y) \to z = (y \to x) \to (y \to z)$ shows that
\[
D(x \wedge y) = D(y \to x) \wedge \left( (y \to x) \to D(y) \right).
\]
As $L(D(x \wedge y)) = L(x \wedge y) = 2$, it is clear that
\[
D(x \to y) \neq (x \to y) \to D(x) \textnormal{ and } D(y \to x) \neq (y \to x) \to D(y).
\]
Therefore, $D(x \to y) \in \{ D(y \to x); (y \to x) \to D(y) \}$. From $x \wedge y = (x \to y)x = (y \to x)y$ it follows that $x \to y = y \to x$ can only appear when $x = y$; this finally shows that $D(x \to y) = (y \to x) \to D(y)$.
\end{proof}

Taking into account \cref{rema:l_algebras_and_cycle_sets}, the theorem can thus be interpreted as saying that each modular noetherian right $\ell$-group with a strong order unit contains an involutive set-theoretic solution to the Yang-Baxter equation. We will follow this line of thought in the following section where we will prove that $G$ contains a distributive noetherian right $\ell$-group, which is equivalent to the structure group of a set-theoretic solution to the Yang-Baxter equation (see \cite[Sections 2 + 3]{Rump-Geometric-Garside}).

\section{The distributive scaffold and the beam decomposition} \label{sec:the_distributive_scaffold}

The aim of this section is to prove the following two results
\begin{enumerate}[1)]
\item The subgroup $H$ generated by $X(\Ccal(G^-))$ in $G$ is a distributive sublattice.
\item The \emph{lattice} decomposition $H \cong \prod_{i=1}^k \mathbb{Z}$ is induced by a \emph{lattice} decomposition $G \cong \prod_{i=1}^k G_i$ by suitable sublattices $G_i \subseteq G$.
\end{enumerate}

Given a discrete $L$-algebra $X$ with duality, we can define the \emph{structure monoid}
\[
M(X) := \left\langle X \setminus \{ e \} | (x \to y) x = (y \to x) y  \right\rangle_{\mathrm{mon}}
\]
and the \emph{structure group}
\[
G(X) := \left\langle X \setminus \{ e \} | (x \to y) x = (y \to x) y  \right\rangle_{\mathrm{gr}}.
\]
There is a natural monoid homomorphism $\iota: M(X) \to G(X)$ that is given by sending $x$ to $x$ whenever $x \in X \setminus \{ e \}$.

One can easily show ( \cite[Proposition 5.]{Rump-Geometric-Garside}) that if $G$ is a modular noetherian right $\ell$-group, then the set $\tilde{X}(G^-)$ is closed under the $\to$-operation: let $x,y \in \tilde{X}(G^-)$ be such that $x,y \neq e$ and $x \neq y$ (all other cases are trivial), then $x \vee y = e$, so $[x \to y, e] \cong [x \wedge y, x] \cong [y,x \vee y]$ is a chain of length $1$. Therefore, $x \to y \in X(G^-)$ as well.

We have the following results proved by Chouraqui and Rump.

\begin{thm}
If $X$ is a finite, discrete $L$-algebra with duality, the homomorphism $\iota: M(X) \to G(X)$ embeds $M(X)$ as the negative cone of a distributive noetherian right-invariant lattice-order on $G(X)$ that has a strong order unit.

For each distributive noetherian right $\ell$-group with strong order unit, $\tilde{X}(G^-) = X(G^-) \cup \{ e \}$ is a finite, discrete $L$-algebra with duality and the canonical homomorphisms $G(\tilde{X}(G^-)) \to G$ and $M(\tilde{X}(G^-)) \to G^-$ are isomorphisms.
\end{thm}

\begin{proof}
The fact that $G = G(X)$ is a right $\ell$-group with $X(G^-) = X$ is proven in \cite[Theorem 3.3]{chouraqui}. The proof that $G$ is a distributive noetherian right $\ell$-group can be found in \cite[Theorem 2, Proposition 7]{Rump-Geometric-Garside}. For the result that each distributive noetherian right $\ell$-group $G$ is isomorphic to the structure group $G(\tilde{X}(G^-))$, see \cite[Theorem 2]{Rump-Geometric-Garside}.
\end{proof}

As $ X:= \tilde{X}(\Ccal(G^-))$ is indeed an $L$-algebra with duality (\cref{thm:atoms_of_center_are_l_algebra_with_duality}), we know that $G(X)$ is a distributive noetherian right $\ell$-group with negative cone $M(X)$. We are already in the position to prove:

\begin{thm} \label{thm:the_distributive_scaffold}
Let $G$ be a modular noetherian right $\ell$-group with the strong order unit $s = \left( \bigwedge X(G^-) \right)^{-1}$. Then there is a natural homomorphism $f: G(\tilde{X}(\Ccal(G^-))) \to G$ that embeds $G(\tilde{X}(\Ccal(G^-)))$ as a distributive sublattice. 
\end{thm}

\begin{proof}
The existence of the homomorphism is clear, as the relations $(a \to b)a = (b \to a)b$ for $a,b \in X := \tilde{X}(\Ccal(G^-))$ are inherited from $G^-$. We are therefore left with showing that this map is also an injective homomorphism of lattices.

We first restrict to $M(X)$. Given two expressions $g_1 = x_kx_{k-1} \ldots x_1$, $g_2 = y_ly_{l-1} \ldots y_1$, with all $x_i, y_j \in X$, we can use \cref{eq:l_structure_left} and \cref{eq:l_structure_right} to express $g_1 \to g_2$ and $g_2 \to g_1$ as products from elements in $X$. If $g_1 = g_2$ holds in $G$, then $g_1 \to g_2 = e = g_2 \to g_1$ and the expression can consist of $e$'s only. Then $g_1 \to g_2 = g_2 \to g_1 = e$ holds in $M(X)$ as well, from which it follows that $g_1 = g_2$ in $M(X)$. By the rule $a \wedge b = (a \to b)a$, which holds both in $M(X)$ as in $G^-$, we conclude that the map $M(X) \to G^-$ is an injective homomorphism of $\wedge$-semilattices.

As $G(X)$ is a group of left fractions for $M(X)$, it is easy to see that the canonical group homomorphism $G(X) \to G$ has trivial kernel.

Therefore, we know that $\mathrm{im}(f)$ is $\wedge$-subsemilattice of $G$ that is a distributive lattice under the partial order inherited by $G$. We are left with showing that $\mathrm{im}(f)$ is a $\vee$-subsemilattice as well.

This can be done easiest as follows: the right $\ell$-group $G^{\sim} = (G(X),\leq_{\mathrm{op}})$ with $g \leq_{\mathrm{op}} h \Leftrightarrow h \leq g$, is still distributive and noetherian, so $X(G^{\sim}) = X^{-1}$ also has the structure of a discrete $L$-algebra $(X^{-1},\rightsquigarrow)$ with duality. Therefore, we have to show that for $x,y \in X(\Ccal(G^-))$, the (unique) element $x \rightsquigarrow y := z \in G$ that fulfills $z^{-1} = (x^{-1} \vee y^{-1})x$ is still in $X(\Ccal(G^-))$ - this is because then, the natural homomorphism $f^{\sim}: G(X^{-1},\rightsquigarrow)) \to (G,\leq_{\mathrm{op}})$, sending $x^{-1}$ to its counterpart in $G$ is an embedding of $(G(X),\leq_{\mathrm{op}})$ as a $\wedge_{\mathrm{op}}$-subsemilattice of $(G,\leq_{\mathrm{op}})$, that is, a $\vee$-subsemilattice of $(G,\leq)$.

So, let $x,y \in X(\Ccal(G^-))$. We may assume that $x,y \neq e$ and $x \neq y$. Using \cref{lem:ba'_in_cent} and the fact that conjugation by $s$ induces a lattice automorphism of $[s^{-1},e]$ (\cref{lem:conjugation_by_s_is_automorphism}), we see that $x^{-1}s^{-1},y^{-1}s^{-1} \in \Ccal(G^-)$ and in $\Ccal(G^-)$ we have the covering relations 
\[
x^{-1}s^{-1},y^{-1}s^{-1} \succ_{\Ccal(G^-)} s^{-1}.
\]
By modularity, this implies $x^{-1}s^{-1} \vee y^{-1}s^{-1} \succ_{\Ccal(G^-)} x^{-1}s^{-1}$; another application of \cref{lem:ba'_in_cent} shows that $x^{-1}s^{-1} (x^{-1}s^{-1} \vee y^{-1}s^{-1})^{-1} \in X(\Ccal(G^-))$. Therefore,
\begin{align*}
X(\Ccal(G^-))^{-1} & \ni \left( x^{-1}s^{-1} (x^{-1}s^{-1} \vee y^{-1}s^{-1})^{-1} \right)^{-1} \\
& = (x^{-1}s^{-1} \vee y^{-1}s^{-1})sx \\
& = (x^{-1} \vee y^{-1})x \\
& = z^{-1} = (x \rightsquigarrow y)^{-1}.
\end{align*}
This concludes the proof.
\end{proof}

\begin{defi}
The \emph{distributive scaffold} is the subgroup $\mathcal{D}(G) := \left\langle \Ccal(G^-) \right\rangle _{\mathrm{gr}} \leq G$.
\end{defi}

In the distributive right $\ell$-group $G(X)$, we can define for each $z \in X \setminus \{ e \}$ a sequence of \emph{frozen elements} $z, D(z)z, D^2(z)D(z)z, \ldots$ (see \cite[2.1]{chouraqui_godelle})

We want to use a more generalized notion:

\begin{defi}
For an element $z \in \Ccal(G^-)$ and an integer $n \geq 1$ we define its \emph{$n$-th frozen power} as
\[
\Phi_n(z) := D^{n-1}(z)D^{n-2}(z) \ldots D(z)z.
\]
\end{defi}

Note that $\Phi_1(z) = z$ for all $z \in \Ccal(G^-)$. Using the fact that $D$ preserves the degrees of central elements (\cref{prop:D_preserves_deg_and_U}), one can see pretty easily that $\deg(\Phi_n(z)) = n \cdot \deg(z)$.

We will soon prove that the frozen powers $\Phi_n(z)$ ($z \in X(\Ccal(G^-))$) together form a dual frame in $[s^{-n},e]$. In our proof strategy we will take advantage of knowing where irreducibles are located:

\begin{lem} \label{lem:meet_irreducibles_over_frozen_powers}
Let $n \geq 1$ and let $g \in G^-$ be meet-irreducible with $g \geq s^{-n}$. Then there is a unique element $z \in X(\Ccal(G^-))$ such that $g \geq \Phi_n(z)$.
\end{lem}

\begin{proof}
Let $g = g_k g_{k-1} \ldots g_1$ be the right-normal factorization, which clearly has $k \leq n$ factors (see \cref{thm:right_normal_facs_exist_and_unique}).
By \cref{prop:meet_irreducibles_are_1_homogeneous}, we know that $\deg(g_i) = 1$ ($1 \leq i \leq k$). As $g_1$ is therefore a dual atom, there must be a unique $z \in X(\Ccal(G^-))$ with $g_1 \geq z$. Fix this $z$.

We will now prove inductively that $H_i := g_i g_{i-1} \ldots g_1 \geq \Phi_i(z)$ for all $1 \leq i \leq k$. Therefore, assume that this has already been proven for some value of $i$. By modularity, the element $H_{i+1} \wedge \Phi_i(z) \prec \Phi_i(z)$. Note that $\lambda(H_{i+1}) = i+1$ implies by \cref{thm:right_normal_facs_exist_and_unique} that $H_{i+1} \not \geq s^{-i}$. We now easily see that
\[
H_{i+1} \wedge \Phi_i(z) \geq s^{-1} \Phi_i(z) = \bigwedge_{\tilde{z} \in X(\Ccal(G^-))} \tilde{z} \cdot \Phi_i(z),
\]
and that $H_{i+1} \wedge \Phi_i(z)$ can only lie in one of these factors. On the other hand, we see that $\lambda( \tilde{z} \cdot \Phi_i(z)) = \lambda(\tilde{z}D^{i-1}(z) \ldots D(z)z) = i+1$ - but due to \cref{prop:duality_gives_right_maximal_expression} this is only possible with $\tilde{z} = D^i(z)$. Therefore, $H_{i+1} \geq \tilde{z} \Phi_i(z) \geq \Phi_{i+1}(z)$.

Having arrived at $g \geq \Phi_k(z)$, it is now also clear that $g \geq \Phi_n(z)$.
\end{proof}

We can now prove the following:

\begin{prop} \label{prop:negative_elements_are_meets_over_frozen_elements}
Let $n \geq 1$, and $\{ z_1,\ldots,z_k \} = X(\Ccal(G^-))$, then each element $g \in [s^{-n},e]$ is expressible as
\[
g = \bigwedge_{i=1}^k (g \vee \Phi_n(z_i)).
\]

In particular, $s^{-n} = \bigwedge_{i=1}^k \Phi_n(z_i)$.
\end{prop}

\begin{proof}
As $G$ is noetherian, we clearly have a finite expression $g = \bigwedge_{j=1}^l \iota_j$, where each $\iota_j$ is meet-irreducible. By \cref{lem:meet_irreducibles_over_frozen_powers}, we can rewrite this term as
\[
g = \bigwedge_{i=1}^k \left( \bigwedge \{ \iota_j : \iota_j \geq \Phi_n(z_i) \} \right) =: \bigwedge_{i=1}^k h_i.
\]
From $h_i \geq g \vee \Phi_n(z_i)$, we now deduce the statement of the proposition.
\end{proof}

As $\deg(s^{-n}) = n \cdot \deg(s^{-1}) = \sum_{i=1}^k n \cdot \deg(z_i) = \sum_{i=1}^k \deg ( \Phi_n(z_i) )$, we can derive:

\begin{prop} \label{prop:phi_n_z_are_a_dual_frame}
The elements $\Phi_n(z_i)$ ($1 \leq i \leq k$) form a dual frame of $[s^{-n},e]$. More precisely, $[s^{-n},e]$ is an internal direct product
\[
[s^{-n},e] = \prod_{i=1}^k \Phi_n(z_i)^{\uparrow}.
\]
In particular, the elements $\Phi_n(z_i)$ are central in $[s^{-n},e]$.
\end{prop}

\begin{proof}
The elements $\Phi_n(z_i)$ ($1 \leq i \leq k$) are clearly dually spanning. If they were not dually independent, we could find $A, B \subseteq \{1, \ldots, k \}$ with $A \cap B = \emptyset$ but
\[ 
\left( \bigwedge_{i \in A} \Phi_n(z_i) \right) \vee \left( \bigwedge_{i \in B}  \Phi_n(z_i) \right) < e.
\]
An application of the parallelogram identity (\cref{lem:parallelogram_identity}) would imply 
\[
\deg \left( \bigwedge_{i \in A \cup B} \Phi_n(z_i) \right) < \sum_{i \in A \cup B} \deg(\Phi_n(z_i));
\]
using the identity again, this leads to
\begin{align*}
\deg(s^{-n}) & =  \deg \left( \bigwedge_{i=1}^k \Phi_n(z_i) \right) \\
& \leq \deg \left( \bigwedge_{i \in A \cup B} \Phi_n(z_i) \right) + \deg \left( \bigwedge_{i \in \{1,\ldots, k \} \setminus (A \cup B)} \Phi_n(z_i) \right) \\
& < \sum_{i \in A \cup B} \deg(\Phi_n(z_i)) + \sum_{i \in \{1,\ldots, k \} \setminus (A \cup B)} \deg(\Phi_n(z_i)) \\
& = \sum_{i = 1}^k \deg(\Phi_n(z_i)) = \deg(s^{-n}).
\end{align*}
A contradiction to the fact that the $\Phi_n(z_i)$ ($1 \leq i \leq k$) actually dually span the interval $[s^{-n},e]$.

By \cref{prop:negative_elements_are_meets_over_frozen_elements}, each $g \in [s^{-n},e]$ has the decomposition
\[
g = \bigwedge_{i=1}^k (g \vee \Phi_n(z_i)) 
\]
Setting $g_i := g \vee \Phi_n(z_i)$, we see from the dual independence of the $\Phi_n(z_i)$ that the elements $g_i \geq \Phi_n(z_i)$ are dually independent either. Therefore, $\deg(g) = \sum_{i=1}^k \deg(g_i)$. In a meet-representation of $g$, the elements $g_i$ ($1 \leq \leq k$) consistute the smallest possible choice, meaning that for a representation $g = \bigwedge_{i=1}^k h_i$ with all $h_i \geq \Phi_n(z_i)$, we necessarily have $h_i \geq g_i$. But then the $h_i$ ($1 \leq i \leq k$) must also be dually independent, implying that $\deg(g) = \sum_{i=1}^k \deg(h_i)$ which is only possible if $\deg(g_i) = \deg(h_i)$ resp. $g_i = h_i$ for all $i$. This shows that the meet-representation is indeed unique, so we have an internal inner product.
\end{proof}

\begin{lem} \label{lem:dual_atoms_over_frozen_elements}
Let $z \in \Ccal(G^-)$ and $n \geq 1$. If $x \in X(G^-)$ is a dual atom with $x \geq \Phi_n(z)$, then already $x \geq z$.
\end{lem}

\begin{proof}
Note that $\Phi_n(z) = D^{n-1}(z) D^{n-2}(z) \ldots D(z)z$ is a right-normal expression by \cref{prop:duality_gives_right_maximal_expression}. This implies $\Phi_n(z) \vee s^{-1} = z$. For $x \in X(G^-)$ we have $x \geq s^{-1}$. Therefore, $x \geq \Phi_n(z) \vee s^{-1} = z$.
\end{proof}

\begin{prop} \label{prop:dual_atoms_of_center_in_strong_order_powers}
For all $n \geq 1$ we have $X(\Cent([s^{-n},e])) = \{ \Phi_n(z) : z \in X(\Ccal(G^-)) \}$.
\end{prop}

\begin{proof}
Letting $X(\Ccal(G^-)) =: \{ z_1 , \ldots, z_k \}$, we already know that the elements $\Phi_n(z_i)$ ($1 \leq i \leq k$) are central and constitute a dual frame of $[s^{-n},e]$. We must show that none of these factors can be decomposed any further. So let $z = z_i$ for some index $i$.

Assume that there was an internal direct product $\Phi_n(z_)^{\uparrow} = L_1 \times L_2$ for some nontrivial sublattices $L_1,L_2 \subseteq \Phi_n(z)^{\uparrow}$, then we had a partition
\[
X(G^-) \cap \Phi_n(z)^{\uparrow} = (X(G^-) \cap L_1) \uplus (X(G^-) \cap L_2)
\]
with both partition classes on the right side nonempty. But $\bigwedge \left( X(G^- \cap \Phi_n(z)^{\uparrow}) \right) = s^{-1} \vee \Phi_n(z)^{\uparrow} = z$ as each dual atom lies above $s^{-1}$ (\cref{cor:all_dual_atoms_are_above_s'}) and $s^{-1} \vee \Phi_n(z)^{\uparrow}$ has to be a meet of dual atoms.

Therefore, we would get a nontrivial internal direct product $z^{\uparrow} = z_1^{\uparrow} \times z_2^{\uparrow}$ with the factors $z_i = \bigwedge (X(G^-) \cap L_i)$ ($i=1;2$), a contradiction to the fact that for $z \in X(\Ccal(G^-))$, the lattice $z^{\uparrow}$ is directly irreducible.
\end{proof}

We now go one step further by considering \emph{semi-beams} as lattice counterparts to frozen elements:

\begin{defi}
Let $z \in \Ccal(G^-)$, then the \emph{semi-beam} associated with $z$ is defined as
\[
\beth^-(z) := \bigcup_{n \geq 1} \Phi_n(z)^{\uparrow}.
\]
\end{defi}

We are now in the fortunate position to prove the following theorem:

\begin{thm} \label{thm:G-_is_product_of_indecomposable_semibeams}
There is an isomorphism of lattices $G^- \cong \prod_{i=1}^k \beth^-(z_i)$, where $\{ z_1, \ldots , z_k \} := X(\Ccal(G^-))$. Furthermore, the lattices $\beth^-(z_i)$ are directly indecomposable.
\end{thm}

\begin{proof}
By \cref{prop:external_products_are_internal_products}, we need to show that each element $g \in G^-$ is uniquely representable as $\bigwedge_{i=1}^k g_i$ with $g_i \in \beth^-(z_i)$.

As $g \in [s^{-n},e]$ for all large enough integers $n$, this follows immediately from \cref{prop:negative_elements_are_meets_over_frozen_elements}.

The indecomposability statement is proven in the same way as in \cref{prop:dual_atoms_of_center_in_strong_order_powers}, by looking at the image of $z_i$ in a product decomposition in $\beth^-(z_i)$.
\end{proof}

Note that $X(\Ccal(G^-)) = \{z_1, \ldots , z_k \} = X(\mathcal{D}(G)^-)$, by construction. Therefore, with regard to the degree function in the distributive scaffold $\mathcal{D}(G)$, we have $\deg_{\mathcal{D}(G)}(z_i) = 1$ ($1 \leq i \leq k$). As $\mathcal{D}(G)$ is a right sub-$\ell$-group of $G$ with $s \in \mathcal{D}(G)$, a factorization $g = g_kg_{k-1}\ldots g_1$ (all $g_i \in \Ccal(G^-)$) is right-normal in $\mathcal{D}(G)$ if and only if it is right-normal in $G$.

The factorizations $\Phi_n(z_i) = D^{n-1}(z_i)\ldots D(z_i)z_i$ are therefore $1$-homogeneous with respect to $\mathcal{D}(G^-)$, from which it follows (\cref{prop:meet_irreducibles_are_1_homogeneous}, \cref{prop:meet_irreducibles_are chains}) that they are dual chains in $\mathcal{D}(G^-)$, i.e. the elements lying above $\Phi_n(z_i)$ are exactly the elements $\Phi_k(z_i)$ with ($k \leq n$). This proves:

\begin{prop} \label{prop:intersect_semibeams_with_dg}
For all $z \in X(\Ccal(G^-))$, we have:
\[
\beth^-_G(z) \cap \mathcal{D}(G) = \{ \Phi_n(z) : n \geq 0 \} = \beth^-_{\mathcal{D}(G)}(z).
\]
In particular, the semi-beam decomposition of $G^-$ induces the semi-beam decomposition of $\mathcal{D}(G)$ whose factors are chains of frozen powers.
\end{prop}

We now want to show that the semi-beam decomposition of $G^-$ is rigid enough in order to be preserved by right-multiplications:

\begin{prop} \label{prop:rigidity_of_semibeams}
Let $g \in G^-$, be decomposed as $g = \bigwedge_{i=1}^kh_i$ with $h_i \in \beth^-(z_i)$.
Then there is a permutation $\pi_g : \{ 1,\ldots,k \} \to \{ 1, \ldots ,k \}$, such that for all $1 \leq i \leq k$:
\[
\beth^-(z_i) \cdot g = \left( \bigwedge_{j \neq i^{\pi_g}} h_j \right) \wedge (h_{i^{\pi_g}}^{\downarrow})_{\beth^-(z_{i^{\pi_g}})}
\]
\end{prop}

Here, the subscript $\beth^-(z_i)$ means that the interval is formed within the factor $\beth^-(z_i)$.

\begin{proof}
Let $g = \bigwedge_{i=1}^k h_i$ with $h_i \in \beth^-(z_i)$. Representing $G^-$ as an external product of semibeams, we have the following lattice maps:
\[
G^- \cong \prod_{i=1}^k \beth^-(z_i) \underset{\sim}{\overset{x \mapsto xg}{\longrightarrow}} \prod_{i=1}^k \beth^-(z_i)g = \prod_{j=1}^k (h_j^{\downarrow})_{\beth^-(z_j)}.
\]
As each $\beth^-(z_i)$ is directly indecomposable \cref{thm:G-_is_product_of_indecomposable_semibeams} and the factorization on the right-hand side has $k$ non-trivial pieces, it follows from the uniqueness theorem for indecomposable factorizations \cite[Corollary 277.]{Graetzer-Lattice}, that the factors can be identified by a permutation. But this is exactly the statement of the proposition.
\end{proof}

We now extend the semibeam decomposition of $G^-$ to a \emph{beam decomposition} of $G$.

Let $G^- = \prod_{i=1}^k \beth^-(z_i)$ be the semibeam decomposition. By right-invariance, we can shift the semibeam decomposition in order to get for each $n \geq 1$ an internal decomposition $s^{n \downarrow} = \prod_{i=1}^k \beth^-(z_i) s^n$.

Writing $e = \bigwedge_{i = 1}^k x_i^{(n)}$ with $x_i^{(n)} \in \beth^-(z_i)s^n$, the image of $G^-$ in the external decomposition of $s^{n \downarrow}$ can be identified with the subset
\[
\prod_{i = 1}^k x_i^{(n)\downarrow} \subseteq  \prod_{i=1}^k \beth^-(z_i) s^n.
\]
Note that the decomposition on the right is clearly an indecomposable factorization, as its $k$ pieces are clearly isomorphic to the semibeams of $G^-$. On the left hand side, we have a factorization into $k$ non-trivial sublattices. Therefore, by \cite[Corollary 277.]{Graetzer-Lattice}, the factorization on the left hand side must be, up to some permutation, come from the factorization of $G^-$ into semibeams; therefore, the factorization $s^{n \downarrow} = \prod_{i=1}^k \beth^-(z_i)s^n$ induces the factorization $G^- \prod_{i=1}^k \beth^-(z_i)$. More generally, the factorization $s^{n \downarrow} = \prod_{i=1}^k \beth^-(z_i)s^n$ induces the factorization $s^{m \downarrow} = \prod_{i=1}^k \beth^-(z_i)s^m$ for all $0 \leq m \leq n$.

As $G = \bigcup_{n \geq 0} s^{n \downarrow}$, we can deduce that there are convex sublattices $\beth(z_i) \subseteq G$, such that $\beth(z_i) \cap G^- = \beth^-(z_i)$ which provide, by the extension described in \cref{prop:glueing_factorizations}, an isomorphism $G \overset{\sim}{\to} \prod_{i=1}^k \beth(z_i)$.

The sublattices $\beth(z_i)$ constructed by the extension procedure are the \emph{beams} (associated with $z_i$).

We state the observation as a theorem:

\begin{thm} \label{thm:G_is_product_of_beams}
There is an isomorphism of lattices $\psi: G \overset{\sim}{\to} \prod_{i=1}^k \beth(z_i)$ where $\{z_1, \ldots , z_k \} := X(\Ccal(G^-))$.
\end{thm}

By \cref{prop:intersect_semibeams_with_dg}, we have the induced internal decomposition
\[
\mathcal{D}(G)^- = \prod_{i=1}^k \beth_{\mathcal{D}(G)}^-(z_i) = \prod_{i=1}^k \beth^-(z_i) \cap G^-.
\]
This again shifts to internal decompositions
\[
\mathcal{D}(G)^- \cdot s^n = \prod_{i=1}^k \beth_{\mathcal{D}(G)}^-(z_i) \cdot s^n = \prod_{i=1}^k \left( \beth^-(z_i) \cap G^- \right) \cdot s^n.
\]
Recall that $s^{-n} = \bigwedge_{i=1}^k \Phi_n(z_i)$, which shows that $e = \bigwedge_{i=1}^k \Phi_n(z_i)s^{-n}$, so $\varepsilon_i^{s^{n \downarrow},G^-} = \Phi_n(z_i)s^{-n} \in \mathcal{D}(G)^-$. Examining the formula defining the extension this shows
\[
\beth^-_{\mathcal{D}(G)}(z_i)^{\mathcal{D}(G)s^n} \subseteq \beth_G^-(z_i)^{Gs^n},
\]
i.e. the extension of semi-beams to beams in $G$ is compatible with the extension process in $\mathcal{D}(G)$. This proves:

\begin{prop} \label{prop:intersect_beams_with_g}
For all $z \in X(\Ccal(G^-))$,
\[
\beth_G(z) \cap \mathcal{D}(G) = \beth_{\mathcal{D}(G)}(z)
\]
and the following diagram commutes:
\begin{center}
\begin{tikzpicture}
\node (dg) at (0,1.5) {$\mathcal{D}(G)$} ;
\node (g) at (0,0) {$G$} ;
\node (dgprod) at (4,1.5) {$\prod_{i=1}^k \beth_{\mathcal{D}(G)}(z_i)$} ;
\node (gprod) at (4,0) {$\prod_{i=1}^k \beth_G(z_i)$} ;

\draw [->] (dg) -- (dgprod) node [midway,below] {$\sim$} node [midway,above] {$\psi_{\mathcal{D}(G)}$} ;
\draw [->] (g) -- (gprod) node [midway,above] {$\sim$} node [midway,below] {$\psi_G$} ;
\draw [right hook-latex] (dg) -- (g) ;
\draw [right hook-latex] (dgprod) -- (gprod) ; 
\end{tikzpicture}
\end{center}
Here, the vertical arrows are inclusions and the horizontal arrows are the maps $\psi$ given in \cref{thm:G_is_product_of_beams}.
\end{prop}
Note that the diagram is simply commutative as $\beth_{\mathcal{D}(G)}(z) \subseteq \beth_G(z)$ for all $z \in X(\Ccal(G^-))$ and the isomorphism $\psi_G$ restricts to $\psi_{\mathcal{D}(G)}$, as they are given by the exact same term.

It also turns out that the beams are as rigid as the semibeams!

\begin{thm} \label{thm:rigidity_of_beams}
Let $\rho_g: G \to G$; $h \mapsto hg$. For all $g \in G^-$, there is a permutation $\pi_g: \{1, \ldots,k \} \to \{ 1, \ldots, k \}$ and lattice isomorphisms $\alpha_i : \beth(z_i) \to \beth(z_{i^{\pi_g}})$ such that for all $\underline{x} = (x_i)_{1 \leq i \leq k} \in \prod_{i=1}^k \beth(z_i)$ we have
\[
\left( \psi \circ \rho_g \circ \psi^{-1} \right) ( \underline{x} ) = ( \alpha_{i^{\pi_g^{-1}}}(x_{i^{\pi_g^{-1}}})  )_{1 \leq i \leq k}
\]
\end{thm}

\begin{proof}
Clearly, $\psi \circ \rho_g \circ \psi^{-1}$ is an automorphism of the lattice $\prod_{i=1}^k \beth(z_i)$. Therefore, it suffices that each automorphism $\varphi : \prod_{i=1}^k \beth(z_i) \to \prod_{i=1}^k \beth(z_i)$ has the desired form.

Note first that $\varphi$ maps downsets $\underline{x}^{\downarrow} = \prod_{i=1}^k x_i^{\downarrow}$ isomorphically to downsets $\varphi(\underline{x}) =: \underline{y} = \prod_{i=1}^k y_i^{\downarrow}$. Recall that all downsets are isomorphic, as lattices, to $G^- \cong \prod_{i=1}^k \beth^-(z_i)$, so the decompositions are already indecomposable decompositions and $\varphi$ therefore just permutes the factors in the sense that we can find a permutation $\pi: \{1, \ldots, k\} \to \{1, \ldots, k\}$ and automorphisms $\alpha_i: x_i^{\downarrow} \to y_{i^{\pi}}^{\downarrow}$ such that
\begin{equation} \label{eq:rigidity_of_downsets}
\varphi ( \underline{w} ) = ( \alpha_{i^{\pi^{-1}}}(w_{i^{\pi^{-1}}})  )_{1 \leq i \leq k}
\end{equation}
for all $\underline{w} \in \underline{x}^{\downarrow}$. When ranging over different downsets $\underline{x}^{\downarrow}$, these maps must restrict to each other under inclusion of downsets which is only possible if $\pi$ is independent of the downset $\underline{x}^{\downarrow}$. Therefore, the formula in \cref{eq:rigidity_of_downsets} is valid on the entire domain $\prod_{i=1}^k \beth(z_i)$.
\end{proof}

We remark here that the beams $\beth(z)$ are not necessarily subgroups. However, they are sublattices of $G$ by construction. However, from the preceding results we can actually derive an easy recipe to find convex subgroups: write
\[
G \cong \prod_{i=1}^{k}\beth(z_i) = \prod_{I \in \mathcal{I}} \left( \prod_{i \in I} \beth(z_i) \right)
\]
where $\mathcal{I}$ is the set of isomorphism classes of beams $\beth(z_i)$ appearing in $G$, and $i,j$ lie in the same class $I \in \mathcal{I}$ if there is a lattice isomorphism $\beth(z_i) \cong \beth(z_j)$. Collecting isomorphic lattice factors, $G$ decomposes as a product of \emph{isotypical components} $G_I = \prod_{i \in I} \beth(z_i)$. Due to rigidity (\cref{thm:rigidity_of_beams}), each right-multiplication $h \mapsto hg$, viewed in the representation $\prod_{i=1}^k \beth(z_i)$, restricts to automorphisms of the isotypical factors $G_I$, which are represented by tuples $(x_i)_{1 \leq i \leq k}$ where $x_i = e$ for $i \not \in I$. Therefore, the isotypical component $G_I$ stabilizes $(e)_{i \not \in I}$ when $G$ acts on the lattice $\prod_{i \not \in I} \beth(z_i)$, that is, by restriction to coordinates not in $I$.

Therefore, the isotypical components correspond to convex subgroups of $G$. Note that this is not necessarily the finest decomposition of $G$ into convex subgroups; however, it is the finest canonical subgroup decomposition you can get from only looking at the lattice structure.

\section{Lattice structure of the beams} \label{sec:lattice_structure_of_beams}

We close our analysis with a lattice-theoretic investigation of the beams.

Let $L$ be a lattice; we call an element $x \in L$ \emph{distributive} if the identity
\[
x \vee (y \wedge z) = (x \vee y) \wedge (x \vee z)
\]
holds for all $y,z \in L$.

With that notion, we can start with definition of a \emph{primary lattice}, which is slightly more general than the one given by Inaba \cite{inaba}.

\begin{defi}
Let $L$ be a modular lattice where every interval has finite length. Then $L$ is called \emph{primary} if any interval $\left[ x,y \right]$ is either a chain or contains no distributive\footnote{Note that in Inaba's original article such elements are called \emph{neutral}. Today, the notion of neutrality is given by another identity which turns out to be equivalent to distributivity whenever $L$ is modular \cite[Chapter III, Section 2.]{Graetzer-Lattice}.} elements aside from $x,y$ when considered as a lattice on its own.
\end{defi}

Using the notion of \emph{perspectivity}, one can give a simple criterion for primarity: if $L$ is bounded from below, recall that two elements $x,y \in L$ are called \emph{perspective} if there is a $z \in L$ with $x \wedge z = y \wedge z = 0$ and $x \vee z = y \vee z$.

\begin{prop} \label{prop:primarity_perspectivity}
If $L$ is modular and of each interval is finite length\footnote{In the cited article, Inaba assumes all lattices to be of finite length, see \cite[page 48]{inaba}. As primarity is a hereditary property, there is no harm in extending the notion to unbounded lattices.} then $L$ is primary if and only if in each interval $\left[ x,y \right]$ ($x,y \in L$), any two atoms are perspective (with respect to $\left[x,y \right]$).
\end{prop}

\begin{proof}
\cite[Theorem 45]{inaba}
\end{proof}

\begin{prop} \label{prop:beams_are_primary}
For each $z \in X(\Ccal(G^-))$, the beam $\beth(z)$ is a primary sublattice of $G$.
\end{prop}

\begin{proof}
By \cref{prop:primarity_perspectivity}, we have to show that for an interval $[a,b] \subseteq \beth(z)$ any two atoms $x,y \in [a,b] \in G$ are perspective. We can assume $x \neq y$; note that $a = x \wedge y$ in this case. With $c = x \vee y$, it suffices to restrict to $[a,c]$. Using right-invariance to shift to some sublattice $[ac^{-1},e] \cong [a,b]$, we can assume that $[ac^{-1},e] \subseteq [\tilde{z},e] \subseteq \beth(\tilde{z})$ - note that we just used the rigidity of beams, \cref{thm:rigidity_of_beams}. As $[\tilde{z},e]$ is an indecomposable modular geometric lattice, $[ac^{-1},e]$ is one as well. In such a lattice, the statement is clearly true.
\end{proof}

We now investigate the existence of dual bases which requires some technical arguments:

\begin{lem} \label{lem:transporting_independence}
Let $L$ be bounded from above and modular and let $\{ a,y_1,\ldots, y_n \} \subseteq L$ be dually independent. If $b \in L$ fulfils $a \vee b = 1_L$, then the set $\{ (a \wedge y_i) \vee b : 1 \leq i \leq n \}$ is dually independent as well.
\end{lem}

\begin{proof}
Set $Y := \bigwedge_{i=1}^n y_i$. The elements $y_1, \ldots, y_n$ are dually independent in $[Y,1_L]$. We also have the diamond isomorphism $[Y,1_L] = [Y, Y \vee a] \overset{\sim}{\to} [Y \wedge a, a]$ which is given by the map $y \mapsto y \wedge a$. It maps the dually independent elements $y_i$ ($1 \leq i \leq n$) from $[Y,1_L]$ to the elements $a \wedge y_i$ ($1 \leq i \leq n$) in $[Y \wedge a, a]$ which are therefore dually independent in $[Y \wedge a, a]$.

A dually independent set keeps being dually independent after replacing some elements by bigger ones, so the elements $(y_i \wedge a) \vee (a \wedge b)$ ($1 \leq i \leq n$) are still dually independent in $[a \wedge Y, a]$. In particular, they are dually independent in $[a \wedge b, a]$.

We have the diamond isomorphism $[a \wedge b, a] \overset{\sim}{\to} [b, a \vee b] = [b,1_L]$ given by $x \mapsto x \vee b$. The images of the dually independent elements $(y_i \wedge a) \vee (a \wedge b)$ ($1 \leq i \leq n$) can be determined as follows:
\[
(y_i \wedge a) \vee (a \wedge b) \mapsto (y_i \wedge a) \vee (a \wedge b) \vee b = (y_i \wedge a) \vee b.
\]
Therefore, the elements $(y_i \wedge a) \vee b$ are dually independent in $[b,1_L]$. In particular, they are dually independent in $L$.
\end{proof}

\begin{lem} \label{lem:coatoms_above_join_of_cochains}
Let $L$ be a bounded from above, modular lattice and $x_1,\ldots, x_n \in L$ dual chains. Setting $X = x_1 \wedge \ldots \wedge x_n$, there are at most $n$ dually independent dual atoms above $X$.
\end{lem}

\begin{proof}
We use induction on $n$: for $n=1$, this is trivial. Now let $n > 1$ and assume that the statement has already been proven for all $k < n$.

Assume to the contrary that there were $n+1$ dually independent dual atoms above $X$, say $y_1,\ldots,y_{n+1}$. Set $X^{\prime} = x_1 \wedge \ldots \wedge x_{n-1}$.

We first show that we can then find $i,j$ with $1 \leq i < j \leq n+1$ such that $(y_i \wedge y_j) \vee X^{\prime} = 1_L$: clearly, not all $y_i$ can be above $X^{\prime}$ since this would contradict the induction hypothesis. Therefore, there is an index $l$ with $1 \leq l \leq n+1$ such that $y_l \ngeq X^{\prime}$. There is another index $m$ with $1 \leq m \leq n+1$ such that $(y_l \wedge y_m) \vee X^{\prime} = 1_L$. Otherwise, for all $m \neq l$, we would have $y_m^{\prime}:= (y_l \wedge y_m) \vee X^{\prime} < 1_L$. Note also that $y_l \ngeq X^{\prime}$ implies $y_m^{\prime} > y_l \wedge y_m$. Since modularity implies $\len([y_l \wedge y_m,1_L]) = 2$, we get for all $m \neq l$ the covering relation $y_m^{\prime} \prec 1_L$.

But then the elements $y_m^{\prime}$ ($m \neq l$) would form a dually independent set of $n$ dual atoms in $[X^{\prime},1_L]$, by \cref{lem:transporting_independence}. This would be a contradiction to the induction hypothesis. So there is an $m \neq l$ with $(y_l \wedge y_m) \vee X^{\prime} = 1_L$. By indexing properly, we can take the desired $i < j$ such that $\left\{ i,j \right\} = \left\{ k,l \right\}$, and the claim is proven.

Set $Y := y_i \wedge y_j$, then
\[
\left[ Y,1_L \right] = \left[ Y, Y \vee X^{\prime} \right] \cong \left[ Y \wedge X^{\prime}, X^{\prime} \right] \subseteq \left[ X,X^{\prime} \right],
\]
where the inclusion follows from $Y \geq X$. But $\left[ X,X^{\prime} \right]$ is a chain, which we can see by the consideration that
\[
\left[ X,X^{\prime} \right] = \left[ X^{\prime} \wedge x_n ,X^{\prime} \right] \cong \left[ x_n, X^{\prime} \vee x_n \right] \subseteq \left[ x_n, 1_L \right].
\]
Therefore, $\left[X, X^{\prime}\right]$ can have at most one dual atom - this, however, would contradict the fact that $[X,X^{\prime}]$ contains an upset isomorphic to $\left[ Y, 1_L \right]$ which has at least two different dual atoms.
\end{proof}

\begin{lem} \label{lem:cochains_independent_when_atoms_are}
Let $L$ be a bounded from above, modular, lattice which fulfils the ascending chain condition. Let $y_1, \ldots y_n \in L$ be dual chains. If $e \succ x_i \geq y_i$ ($1 \leq i \leq n$) are the unique dual atoms lying over the $y_i$ then $y_1,\ldots,y_n$ are dually independent if and only if $x_1,\ldots, x_n$ are. 
\end{lem}

\begin{proof}
If, under the stated conditions, the elements $y_1, \ldots , y_n$ are dually independent, then the elements $x_1, \ldots, x_n$ are dually indepedent as well, since dual independence is preserved under replacing elements by possibly larger ones.

Now assume that $x_1, \ldots, x_n$ are dually independent. We prove that for any $I \subseteq \{ 1, \ldots, n \}$ and $1 \leq j \leq n$ such that $j \notin I$ we have $Y:=\left( \bigwedge_{i \in I} y_i \right) \vee y_j = 1_L$. By \cite[Theorem 360.]{Graetzer-Lattice}, this is equivalent to dual independence.

If we had $Y < 1_L$, then, in particular $y_j \leq Y < 1_L$ which would imply $Y \leq x_j$. With $X := \bigwedge_{i \in I} y_i$, this would imply that there are at least $|I|+1$ dually independent dual atoms above $X$, namely the elements $x_i$ ($i \in I$) and $x_j$, thus contradicting \cref{lem:coatoms_above_join_of_cochains}.
\end{proof}

We now show that meet-irreducibles in $G^-$ can always be extended to larger ones:

\begin{lem} \label{lem:extend_meet_irreducibles}
Let $g \in G^-$ be meet-irreducible. Then there is a meet-irreducible element $g^{\prime} \in G^-$ with $g^{\prime} \prec g$.
\end{lem}

\begin{proof}
Let $g = g_k g_{k-1} \ldots g_1$ be the right-normal factorization. By \cref{prop:meet_irreducibles_are_1_homogeneous}, all $\deg(g_i) = 1$ ($1 \leq i \leq k$). As $s^{-1} = \bigwedge X(G^-)$ and $s^{-1} < s^{-1}g_k^{-1}$, there is an element $g_{k+1} \prec e$ with $g_{k+1} \vee s^{-1}g_k^{-1} = e$. Therefore, $g_{k+1}g_kg_{k-1} \ldots g_1$ is a right-normal expression with all $\deg(g_i) = 1$ ($1 \leq i \leq k+1$) and therefore represents an element $g^{\prime} \prec g$, which is meet-irreducible due to another application of \cref{prop:meet_irreducibles_are_1_homogeneous}.
\end{proof}

\begin{thm} \label{thm:phi_n_z_upset_has_dual_basis}
Let $z \in X(\Ccal(G^-))$, then for each $n \geq 1$, the lattice $ \Phi_n(z)^{\uparrow}$ has a dual basis $y_1, \ldots, y_d$ with $\deg(y_i) = n$ ($1 \leq i \leq d$), where $d := \deg(z)$.
\end{thm}

\begin{proof}
Note that $\deg(z)$ is also the dimension of $[z,e]$ as a geometric lattice. By \cref{prop:dimension_of_geometric_lattice_is_size of_dual_basis}, there is a dual basis $x_1, \ldots , x_d$ for $[z,e]$. By \cref{lem:extend_meet_irreducibles}, for all $1 \leq i \leq d$ we can find meet-irreducibles $y_i \leq x_i$ with $\deg(y_i) = n$. Clearly, all $y_i \geq s^{-n}$. By \cref{lem:meet_irreducibles_over_frozen_powers}, we can conclude that $y_i \geq \Phi_n(z)$.

Due to \cref{lem:cochains_independent_when_atoms_are}, the elements $y_1, \ldots, y_d$ are dually independent, as $x_1, \ldots, x_d$ are. Therefore,
\[
n \cdot d = \deg(\Phi_n(z)) \geq \deg \left( \bigwedge_{i=1}^d y_i \right) = \sum_{i=1}^d \deg(y_i) = d \cdot n. 
\]
So the inequality is actually an equality which implies that $\bigwedge_{i=1}^d y_i = \Phi_n(z)$. Thus, the elements $y_1, \ldots , y_d$ form a dual basis of $\Phi_n(z)^{\uparrow}$.
\end{proof}

All rings considered in the following are unital.

Let $R$ be a ring and ${}_RM$ a left $R$-module. We denote by $L({}_RM)$ the lattice of all $R$-submodules ordered by inclusion. Given $A,B \in R$, the respective lattice-operations are then given by $A \vee B = A + B$ and $A \wedge B = A \cap B$. As it is more comfortable to read, we will denote the lattice-operations as $ +, \cap$. It is well-known that the lattices $L({}_RM)$ are modular \footnote{In fact, the name has its origin in the fact that the prototypes of \emph{modular} lattices are coming from \emph{modules} in the described way.}.

As we are working with free modules most of the time, it is convenient to write $L(R,\delta) := L({}_RR^{\delta})$.

The two types of rings that we will mainly consider are the following:

\begin{defi}{\cite[Definition 6.6]{monk_jonsson}}
A ring $R$ is called \emph{completely primary uniserial} (\emph{cpu}, for short) if it is local and every left- or right ideal $I \leq R$ is of the form $\mathfrak{m}^i$ for some integer $i$, where $\mathfrak{m}$ is the unique maximal ideal of $R$.

The \emph{length} of a cpu ring $R$ is defined as the minimal integer $i$ with the property that $\mathfrak{m}^i = (0)$.

An element $\pi \in \mathfrak{m} \setminus \mathfrak{m}^2$ is called a \emph{uniformizer} of $R$.
\end{defi}

\begin{defi}
Let $Q$ be a (possibly skew) field. A surjective map $v: Q \to \mathbb{Z} \cup \{ + \infty \}$ is a \emph{discrete valuation} if the following three axioms are fulfilled for $x,y \in Q$:
\begin{align*}
v(x) = \infty & \Leftrightarrow x = 0 \\
v(x+y) & \geq \min \{ v(x),v(y) \} \\
v(xy) & = v(x) + v(y).
\end{align*}
A pair of the form $(Q,v)$, with $Q$ a skew field and $v$ a discrete valuation on $Q$, is called a \emph{discrete valuation field} (\emph{dvf}, for short).

If $(Q,v)$ is a discrete valuation field, the subring $R := \{ x \in Q : v(x) \geq 0 \}$ is called a (noncommutative) \emph{discrete valuation ring} (\emph{dvr}, for short).

A \emph{uniformizer} of a dvr is any element $\pi \in R$ with $v(\pi) = 1$.
\end{defi}

When talking about discrete valuation fields or -rings, we will not explicitly mention the valuation $v$. If we are only interested in the dvr $R$, we will also occasionally not mention the dvf $Q$ wherein it is embedded.

As in the commutative theory, every left- or right ideal $I \leq R$ is actually both-sided and of the form $I = \pi^k R$ for some integer $k > 0$, where $\pi$ is an arbitrary uniformizer. More generally, any finitely generated left $R$-submodule $M \leq {}_RQ$ is of the form $M = \pi^k R$ for some $k \in \Z$.

Furthermore, note that $\mathfrak{m} = \pi R$ is the unique maximal ideal of $R$, so $R$ is a local ring.

We furthermore need the notion of an \emph{$R$-lattice} in a module. It should always be clear that when talking about $R$-lattices, we mean the module-theoretic notion, \emph{not} the order-theoretic notion of lattices.

\begin{defi}
Let ${}_RM$ be a left $R$-module over the ring $R$. A submodule $A \in L({}_RM)$ is called an \emph{$R$-lattice} in ${}_RM$, if
\begin{enumerate}[1)]
\item $A$ is finitely generated, and
\item $A$ is an \emph{essential} submodule in ${}_RM$, meaning that it has nonzero intersection with each nonzero submodule in ${}_RM$.
\end{enumerate}
\end{defi}

The totality of all $R$-lattices in ${}_RM$ will be denoted as $\Lat({}_RM)$. If $M = {}_RR^{\delta}$, we will also use the more comfortable notation $\Lat(R,\delta) := \Lat({}_RR^{\delta})$.

It is easy to see that if $R$ is noetherian, the property of being an $R$-lattice in ${}_RM$ is stable under taking sums and intersections of submodules, therefore $\Lat({}_RM)$ is a sublattice of $L({}_RM)$, in this case. If $R$ is a dvr with respective dvf $Q$, then the following characterization is quite useful:

\begin{lem} \label{lem:characterization_of_lattices}
Let $R$ be a dvr with dvf $Q$, $\pi \in R$ a uniformizer thereof, and $\delta \geq 1$ some integer. Then:
\begin{enumerate}[i)]
\item An element $A \leq L(R, \delta)$ is an $R$-lattice if and only if there is an integer $n$ such that $\pi^nR^{\delta} \subseteq A$.
\item An element $A \leq L({}_RQ^{\delta})$ is an $R$-lattice if and only if there is an integer $n$ such that $\pi^nR^{\delta} \subseteq A \subseteq \pi^{-n} R^{\delta}$.
\end{enumerate}
\end{lem}

\begin{proof}
Denote the standard basis of ${}_RR^{\delta}$ resp. ${}_RQ^{\delta}$ by $e_1, \ldots, e_{\delta}$.

Let $A \in L(R, \delta)$ be an $R$-lattice. As $A$ is essential in $R^{\delta}$, the intersections $A \cap Re_i \neq 0$ ($1 \leq i \leq \delta$); therefore $A \cap Re_i = \pi^{n_i}Re_i$ for some $n_i \geq 0$. Taking $n = \max \{ n_i : 1 \leq i \leq \delta \}$, we see that $\pi^nRe_i \subseteq A$ for all $1 \leq i \leq \delta$. Therefore, $\pi^n R^{\delta} \subseteq A$. On the other hand, each submodular $\pi^nR^{\delta}$ is clearly essential in ${}_RR^{\delta}$; therefore any submodule $A$ containing some $\pi^nR^{\delta}$ is essential as well. As we are working with(in) finitely generated modules over noetherian rings, finite generation is already guaranteed!

Now we prove the second statment: the essentiality of an element $A \in L({}_RQ^{\delta})$ is equivalent to the existence of an $n \geq 0$ with $\pi^nR^{\delta} \subseteq A$; this is proved in the same way as above. As a submodule of the form $\pi^{-n}R^{\delta}$ is finitely generated and $R$ is noetherian, every $A$ contained in some submodule of this form is finitely generated as well. On the other hand, let $A \in L({}_RQ^{\delta})$ be generated by the (finitely many) elements $v_j = \sum_{i}q_{ji}e_i$ ($1 \leq j \leq m$). Then there is an $n$ such that all coefficients $q_{ji} \in \pi^{-n}R$. In particular, all $v_j \in \pi^{-n}R^{\delta}$, therefore $A \subseteq \pi^{-n}R^{\delta}$. Of course, the bounds $\pi^{\pm n}R^{\delta}$ are not guaranteed to be \emph{a priori} equal; however, by chosing large enough values you can clearly find an $n$ such that both bounds are satisfied.
\end{proof}

Let $R$ be a dvr and $\pi \in R$ a uniformizer, then it is easily seen that the rings $R_n := R/\pi^nR$ are cpu-rings of length $n$ for all $n \geq 1$.

We furthermore need some notation in order to translate morphisms of modules into morphisms of lattices: let $R,S$ be (unital) rings and ${}_RM, {}_SN$ be left modules. A map $\varphi: M \to N$ is called \emph{$\alpha$-semilinear} for a ring isomorphism $\alpha: R \to S$ if $\varphi$ is a homomorphism of abelian groups and for all $r \in R, m \in M$ we have $\varphi(rm) = \alpha(r)\varphi(m)$. If we are only interested in semilinearity of a $\varphi$ without the need to emphasize the isomorphism $\alpha$ itself, we will simply speak of \emph{semilinear} maps.

Let $\varphi: {}_RM \to {}_SN$ be a semilinear map between modules. Then we have canonical maps
\begin{align*}
\varphi_{\ast}: L({}_RM) \to L({}_SN) \quad ; & \quad A \mapsto \varphi(A) \\
\varphi^{\ast}: L({}_SN) \to L({}_RM) \quad ; & \quad B \mapsto \varphi^{-1}(B).
\end{align*}
These maps can easily be checked for well-definedness. Without further restrictions, we can only say that $\varphi_{\ast}$ is a monomorphism of $\vee$-semilattices and $\varphi^{\ast}$ is a monomorphism of $\wedge$-semilattices.

Note, however, that if $\varphi$ is injective (resp. surjective), the maps $\varphi_{\ast}$ (resp. $\varphi^{\ast}$) are indeed homomorphisms of lattices. In the first case, $L({}_RM)$ is embedded as a downset of $L({}_SN)$; in the second case, $L({}_SN)$ is embedded as an upset of $L({}_RM)$. If $\varphi$ is bijective - that is, a \emph{semilinear isomorphism} - this shows that $\varphi_{\ast}, \varphi^{\ast}$ are isomorphisms of lattices that are mutually inverse.

We remark that it is quite useful to extend this no(ta)tion to the following case: given a \emph{surjective} ring homomorphism $\alpha : R \to S$, we can still define semilinear maps by the rule $\varphi(rm) = \alpha(r)\varphi(m)$. Letting $r$ act on $N$ via scalar multiplication by $\alpha(r)$, we can regard ${}_SN$ as an $R$-module ${}_RN$ in a way that $L({}_RN) = L({}_SN)$. An $\alpha$-semilinear map ${}_RM \to {}_SN$ then is actually the same as a standard homomorphism (or an $\mathrm{id}_R$-semilinear map, if you like fancy words) of $R$-modules ${}_RM \to {}_RN$.

We can now show:

\begin{prop} \label{prop:lat_r_delta_as_a_direct_limit}
Let $R$ be a noncommutative dvr and $\pi \in R$ a uniformizer; furthermore, set $R_n = R/\pi^nR$.

For $n \geq m$, let $\varphi_{n,m}: R_n^{\delta} \to R_m^{\delta}$ be the canonical projection that is $\alpha_{n,m}$-semilinear, where $\alpha_{n,m}: R_n \to R_m$ is the canonical projection. Then the direct system of lattices $\left( L(R_n,\delta) \right)_{n \geq 1}$ with the homomorphisms $\varphi_{n,m}^{\ast}: L(R_m,\delta) \to L(R_n,\delta)$ has the limit:
\[
\Lat(R,\delta) \cong \lim_{\rightarrow} L(R_n,\delta)
\]
\end{prop}

\begin{proof}
Let $A \in L(R,\delta)$ be a lattice. By \cref{lem:characterization_of_lattices}, there is an $n \geq 1$ with $\pi^n R^{\delta} \subseteq A$. Letting $\varphi_n: R^{\delta} \to R_n^{\delta}$ which is $\alpha_n$-semilinear with respect to the canonical projection $\alpha_n: R \to R_n$, we have $A = \varphi_n^{\ast}(A + \pi^n R^{\delta})$. Therefore,
\[
\Lat(R,\delta) = \bigcup_{n = 1}^{\infty}  \varphi_n^{\ast} \left( L(R_n,\delta) \right),
\]
which can now easily identified as $\lim_{\rightarrow} L(R_n,\delta)$, as all $\varphi_n^{\ast}$ are embeddings into $\Lat(R,\delta)$ that are compatible with the direct system in the sense that $\varphi_m^{\ast} = \varphi_n^{\ast} \varphi_{n,m}^{\ast}$.
\end{proof}

The following characterization of dvr's will be useful:

\begin{lem} \label{lem:gubareni_lemma}
Let $R$ be a (potentially noncommutative) local ring, such that the maximal ideal $\mathfrak{m}$ is of the form $\mathfrak{m} = \pi R = R \pi$ for some non-nilpotent $\pi \in R$. If $\bigcap_{n \geq 1} \mathfrak{m}^n = (0)$, then $R$ is a dvr. 
\end{lem}

\begin{proof} 
\cite[Proposition 4.]{gubareni}.
\end{proof}

\begin{lem} \label{lem:inverse_limits_of_cpu_rings_are_dvrs}
Let $(R_k)_{k \geq 1}$ be an inverse system of cpu rings such that the ring $R_k$ has length $k$ and the ring homomorphisms $\varphi_{k,l}: R_k \to R_l$ are all surjective. Then $R = \lim_{\leftarrow} R_k$ is a dvr.
\end{lem}

\begin{proof}
Denote the maximal ideal of $R_k$ by $\mathfrak{m}_k$. It is clear that for $l \leq k$, the ideal $\mathfrak{m}_k^l$ is the only one, such that $R_k/\mathfrak{m}_k^l$ is of length $l$. Therefore, $\ker(\varphi_{k,l}) = \mathfrak{m}_k^l$.

Furthermore, $\varphi_{k,l}^{-1}(\mathfrak{m}_l) = \mathfrak{m}_k$ because the preimage of $\mathfrak{m}_l$ must the the (unique) maximal ideal of $R_k$. Due to surjectivity, it follows that $\varphi_{k,l}(\mathfrak{m}_k) = \varphi_{k,l}(\varphi_{k,l}^{-1}(\mathfrak{m}_l)) = \mathfrak{m}_l$. Therefore, the ideals $(\mathfrak{m}_k)_{k \geq 1}$ form by restriction an inverse system connected by surjections.

Represent the elements of the inverse limit $R$ as $(x_k)_{k \geq 1}$ with $x_k \in R_k$ and $\varphi_{k,l}(x_k) = x_l$ ($k \geq l$). We define the ideal
\[
\mathfrak{m} = \{ (x_k)_{k \geq 1} : x_k \in \mathfrak{m}_k \} \leq R.
\]
Each element $(x_k)_{k \geq 1} \in R \setminus \mathfrak{m}$ is invertible, which can be shown by the following argument: by the preceding arguments, all components $x_k \not \in \mathfrak{m}_k$ and are therefore invertible. As each $x_k$ is invertible in $R_k$, the element $(x_k)_{k \geq 1}$ is invertible in $R$. Therefore, $R$ is local.

Let $\pi_2$ be a uniformizer of $R_2$. Then, by the surjectivity of the connecting maps, there exists an element $\pi = (\pi_k)_{k \geq 1} \in R$ whose component at the index $2$ is $\pi_2$. As $\pi_2 \in \mathfrak{m}_2 \setminus \mathfrak{m}_2^2$, it is clear that $\pi_k \in \mathfrak{m}_k \setminus \mathfrak{m}_k^2$ for all $k \geq 2$. As $R_k$ has length $k$, this implies that $\pi_k^l \neq 0$ for $l < k$. In particular, $\pi^l \neq 0$ holds in $R$ for all $l$, so $\pi$ is non-nilpotent.

As $\mathfrak{m}_k^n = (0)$ holds whenever $n \geq k$, we also see that $\bigcap_{n \geq 1} \mathfrak{m}^n = (0)$.

Finally, let $x = (x_k)_{k \geq 1} \in \mathfrak{m}$. We need to prove the existence of an element $y = (y_k)_{k \geq 1} \in R$ with $\pi y = x$. For each $k \geq 1$, chose $r_k \in R_k$ with $\pi_k r_k = x_k$. It follows for any $k \geq 2$ (!) that
\[
\pi_k r_k = x_k = \varphi_{k+1,k} (\pi_{k+1} r_{k+1}) = \pi_k \underbrace{\varphi_{k+1,k}(r_{k+1})}_{=: y_k}.
\]
It follows that $\pi_k(r_k - y_k) = 0$, so $r_k - y_k \in \mathfrak{m}_k^{k-1} = \ker(\varphi_{k,k-1})$, therefore $\varphi_{k,k-1}(y_k) = \varphi_{k,k-1}(r_k) = y_{k-1}$. It follows that $(y_k)_{k \geq 1} \in R$ and, by construction, $\pi y = x$. The existence of a $y^{\prime} \in R$ with $x= y^{\prime}\pi$ follows by symmetry.

By \cref{lem:gubareni_lemma}, $R$ is a dvr.
\end{proof}

\begin{thm}[\textbf{Inaba}]
Let $L$ be a bounded primary lattice that has a basis $y_1,\ldots , y_{\delta}$ where $\delta \geq 4$ and $\len([0,y_i]) = n_i$ for some positive integers $n_i$ ($1 \leq i \leq \delta$). Then there is a cpu ring $R$ of length $n = \max \{ n_i : 1 \leq i \leq \delta \}$ with uniformizer $\pi$ such that $L \cong L({}_RM)$, with the module
\[
M = \sum_{i = 1}^{\delta} \pi^{n - n_i} R \subseteq R^{\delta}.
\]

In particular, if $\len([0,y_i]) = n$ ($1 \leq i \leq \delta$) for some common integer $n \geq 1$, then $L \cong L(R, \delta)$ for a cpu-ring $R$ that is of length $n$.
\end{thm}

Now Inaba's theorem tells us the following:

\begin{prop} \label{prop:piecewise_coordinatization}
Let $z \in X(\Ccal(G^-))$ have $\delta_z := \deg(z) \geq 4$, then for each $n \geq 1$, there is a cpu-ring $R_{z,n}$ of length $n$ such that there is a homomorphism of lattices $\Phi_n(z)^{\uparrow} \cong L(R_{z,n},\delta_z)$.
\end{prop}

\begin{proof}
By \cref{thm:phi_n_z_upset_has_dual_basis}, the lattice $\Phi_n(z)^{\uparrow}$ has a dual basis of dual chains $y_1,\ldots y_{\delta_z}$ where $\deg(y_i) = n$ ($1 \leq i \leq \delta_z$). By \cref{lem:reflection_of_frames}, there is a basis $Y^{\prime} = y_1^{\prime},\ldots, y_{\delta_z}^{\prime}$ for $\Phi_n(z)^{\uparrow}$ that consists of chains such that $\len([\Phi_n(z),y_i^{\prime}]) = \len([y_i,e]) = n$ for all $i$. The rest follows from Inaba's theorem.
\end{proof}

\begin{defi}
For a $z \in X(\Ccal(G^-))$, the integer $\delta_z := \deg(z)$ will be called the \emph{dimension} of the semibeam $\beth^-(z)$ resp. the beam $\beth(z)$.
\end{defi}

Let $L$ be a lattice. For an element $x \in L$ such that $x^{\downarrow}$ fulfills the ascending chain condition, we define the \emph{radical} of $x$ as
\[
\Rad_L(x) := \bigwedge \{ y \in L : y \prec x \}
\]
and if $x^{\uparrow}$ fulfills the descending chain condition, we define the \emph{socle} of $x$ as
\[
\Soc_L(x) := \bigvee \{ y \in L : y \succ x \}.
\]

\begin{rema}
Note that the socle of a module ${}_RM$ is the sum of all simple submodules in $M$; therefore the module-theoretic socle is represented in $L({}_RM)$ by the element $\Soc_{L({}_RM)}(0)$, whereas the module-theoretic radical, being defined as the intersection of all maximal submodules of $M$, is represented by the element $\Rad_{L({}_RM)}(M)$.
\end{rema}

The following lemma shows that right- and left-normal factorizations in modular noetherian right $\ell$-groups with strong order unit can be seen as a Garside-theoretic counterpart to socle and radical series in module theory.

\begin{lem} \label{lem:rad_and_soc_from_normal_factorizations}
For $g \in G^-$, let $g = g_k g_{k-1} \ldots g_1$ be the right-normal and $g = h_1 h_2 \ldots h_k$ the left-normal factorization. Then
\[
\Rad_{[g,e]}(e) = g_1 \quad ; \quad \Soc_{G^-}(g) = h_2 \ldots h_k.
\]
\end{lem}

\begin{proof}
As $g_1 \in [s^{-1},e]$, the latter being dually atomistic by Rump's theorem, \cref{thm:strong_order_interval_is_geometric}, we know that $\Rad_{[g,e]}(e) \leq g_1$. As $\Rad_{[g,e]}(e)$ is a meet of dual atoms, it follows from \cref{cor:all_dual_atoms_are_above_s'} that $\Rad_{[g,e]}(e) \geq s^{-1}$. As $\Rad_{[g,e]}(e) \geq g$ holds trivially, we have shown that $\Rad_{[g,e]}(e) \geq g \vee s^{-1} = g_1$. This proves the first statement.

From the first statement, we deduce that, with respect to the dual order, we have
\[
\Rad_{[g,e]_{\lesssim}}(e) = h_1,
\]
where $h_1$ is the first factor in the left-normal factorization. This is only the first statement rewritten in terms of the dual order. Using this, we can calculate
\[
\Soc_{G^-}(g) = \Soc_{[g,e]}(g)  = \Soc_{[e,g^{-1}]}(e) \cdot g  = \left( \Rad_{[g,e]_{\lesssim}}(e) \right)^{-1} g  = h_1^{-1}g = h_2 \ldots h_k.
\]
\end{proof}

With this it is easy to determine socle and radical of frozen powers:

\begin{lem} \label{lem:soc_and_rad_of_frozens}
Let $z \in X(\Ccal(G^-))$ and $n \geq 1$. Then $\Rad_{[\Phi_n(z),e]}(e) = z$ and $\Soc_{G^-}(\Phi_n(z)) = \Phi_{n-1}(z)$.
\end{lem}

\begin{proof}
We first note that we have the right-normal factorization
\[
\Phi_n(z) = D^{n-1}(z)D^{n-2}(z) \ldots D(z)z.
\]
Therefore, the first statement follows directly from \cref{lem:rad_and_soc_from_normal_factorizations}.

By \cref{prop:D_preserves_deg_and_U}, this right-normal factorization is $\deg(z)$-homogeneous and therefore also left-normal, by an application of \cref{prop:homogeneous_rightnormal_implies_leftnormal}. Therefore, \cref{lem:rad_and_soc_from_normal_factorizations} implies that
\[
\Soc_{G^-}(\Phi_n(z)) = D^{n-2}(z) \ldots D(z)z = \Phi_{n-1}(z).
\]
\end{proof}

We need another lemma regarding the behavior of $\Rad$ and $\Soc$ in $\beth(z)$:

\begin{lem} \label{lem:soc_and_rad_are_inverse_on_beams}
Let $z \in X(\Ccal(G^-))$, then the maps $g \mapsto \Rad_{\beth(z)}(g)$ and $g \mapsto \Soc_{\beth(z)}(g)$ are mutually inverse automorphisms of the lattice $\beth(z)$.
\end{lem}

\begin{proof}
By right-invariance, together with \cref{prop:join_of_atoms_is_strong_order_unit}, we know that $\Rad_{G}(g) = s^{-1}g$ and $\Soc_{G}(g) = sg$ for all $g \in G$; therefore, $\Rad$ and $\Soc$ are indeed mutally inverse automorphisms of $G$.

Under the isomorphism $\psi : G \to \prod_{i=1}^k \beth(z_i)$ provided by \cref{thm:G_is_product_of_beams}, we have the equalities $\psi(\Rad_{G}(g)) = (\Rad_{\beth(z_i)}(\psi_i(g)))_{1 \leq i \leq k}$ and $\psi(\Soc_G(g)) = (\Soc_{\beth(z_i)}(\psi_i(g)))_{1 \leq i \leq k}$. Therefore, for each $1 \leq i \leq k$, the maps $\Soc_{\beth(z_i)}$ are $\Rad_{\beth(z_i)}$ are order-preserving bijections that are mutually inverse, which implies that they are actually automorphisms of the lattice $\beth(z_i)$.
\end{proof}

Furthermore, we have the following easy results for socles and radicals in submodule lattices:

\begin{lem} \label{lem:soc_and_rad_over_a_local_ring}
Let ${}_RM$ be a left module over the discrete valuation ring $R$ with the uniformizer $\pi$. Let $A \in L({}_RM)$. Then:
\begin{enumerate}[i)]
\item If $A^{\downarrow}$ fulfills the ascending chain condition, then $\Rad_{L({}_RM)}(A) = \pi A$
\item If $A^{\uparrow}$ fulfills the descending chain condition, then $\Soc_{L({}_RM)}(A) = \pi^{-1}A = \{ x \in M : \pi x \in A \}$.
\end{enumerate}
\end{lem}

\begin{proof}
Note that $A \prec B$ holds for $A,B \in L({}_RM)$ if and only if $A/B \cong R/\pi R$ as $R$-modules, which implies that $\pi A \subseteq B$ in this case.

In the first case, this shows that $\pi A \subseteq \Rad_{L({}_RM)}(A)$. On the other hand, $A / \pi A$ can be considered as a left vector space over $R/ \pi R$ and therefore is semisimple. This shows that $\Rad_{L({}_RM)}(A) \subseteq \pi A$, from which we conclude the desired equality.

In the second case, we can argue similarly that $\pi \cdot \Soc_{L({}_RM)}(A) \subseteq A$ which shows the inclusion $\Soc_{L({}_RM)}(A) \subseteq \pi^{-1}A$. On the other hand, considering $\pi^{-1}A / A$ as a left $R/\pi R$-vector space, we can again conclude that $\pi^{-1}A / A$ is a sum of simple modules which implies that $\pi^{-1}A \subseteq \Soc_{L({}_RM)}(A)$; again, this proves equality.
\end{proof}

\begin{lem} \label{lem:soc_and_rad_are_inverse_in_lat}
Let $Q$ be a dvf with dvr $R$ and let $\delta \geq 1$. Then the maps $x \mapsto \Soc(x)$ and $x \mapsto \Rad(x)$ are lattice automorphisms of $\Lat({}_RQ^{\delta})$ that are inverse to each other.
\end{lem}

\begin{proof}
First observe that $\Lat({}_RQ^{\delta})$ is a convex sublattice of $L({}_RQ^{\delta})$. Let $A \in \Lat({}_RQ^{\delta})$. By \cref{lem:soc_and_rad_over_a_local_ring}, $\Rad_{L({}_RQ^{\delta})}(A) = \pi A$, which is again an  $R$-lattice by statement 2 of \cref{lem:characterization_of_lattices}. This implies that $\pi A$ is the intersection of elements $\prec A$ which are $R$-lattices by the same argument. As the intersection is taken in a convex sublattice of $L({}_RQ^{\delta})$, the radical with respect to $\Lat({}_RQ^{\delta})$ cannot be properly contained in the radical with respect to $L({}_RQ^{\delta})$. Therefore, $\Rad_{\Lat({}_RQ^{\delta})}(A) = \pi A$.

By a similar reasoning, we get $\Soc_{\Lat({}_RQ^{\delta})}(A) = \pi^{-1} A$. As multiplication by $\pi$ is invertible in ${}_RQ^{\delta}$, we see that $\Rad$ and $\Soc$ are actually inverse to each other.
\end{proof}

\begin{thm} \label{thm:large_enough_semibeams_are_coordinatizable}
Let $z \in X(\Ccal(G^-))$ have $\delta_z := \deg(z) \geq 4$, then there is a noncommutative dvr $R_z$ such that there is a lattice isomorphism
\[
\beth(z)^- \cong \Lat(R_z, \delta_z).
\]
\end{thm}

\begin{proof}
As we are only working with one fixed $z$, we abbreviate $R_z =: R$, $R_{z,n} =: R_n$, $\delta_z =: \delta$.

We start with the piecewise coordinatizations $c_n : \Phi_n(z)^{\uparrow} \overset{\sim}{\to} L(R_n,\delta)$, that exist by \cref{prop:piecewise_coordinatization}. For any $n \geq 1$, the embedding $\iota_{n+1,n} : \Phi_n(z)^{\uparrow} \hookrightarrow \Phi_{n+1}(z)^{\uparrow}$ corresponds to an embedding of $L(R_n,\delta)$ as an upper set of $L(R_{n+1},\delta)$. Our first step will consist of showing that this embedding is actually induced by a semilinear surjection $R_{n+1}^{\delta} \to R_n^{\delta}$.

We start by determining some socles: by \cref{lem:soc_and_rad_of_frozens}, $\Soc_{G^-}(\Phi_{n+1}(z)) = \Phi_{n}(z)$. Furthermore $\Soc_{L(R_{n+1},\delta)}(0) = \pi_{n+1}^nR_{n+1}^{\delta}$, by \cref{lem:soc_and_rad_over_a_local_ring}.

We define $S_n := R_{n+1}/\pi_{n+1}^nR_{n+1}$. Furthermore, let $p: R_{n+1}^{\delta} \to S_n^{\delta}$ be the canonical projection that is $\beta$-semilinear with respect to the projection $\beta: R_{n+1} \to S_n$. By the preceding discussion, $p^{\ast}$ identifies $L(S_n,\delta)$ as the upset $\Soc_{L(R_{n+1},\delta)}(0)^{\uparrow} \subseteq L(R_{n+1},\delta)$.

We have exactly the same situation within $G^-$: As $\Phi_n(z)$ is the socle of $\Phi_{n+1}(z)$, there is an isomorphism $f: \Phi_n(z)^{\uparrow} \overset{\sim}{\to} L(S_n,\delta)$ fitting into the commutative diagram below: 

\begin{center}
\begin{tikzpicture}
\node (Ln) at (0,0) {$L(R_n,\delta)$} ;
\node (Lsoc) at (5,0) {$L(S_n,\delta)$} ;
\node (Ln+1) at (10,0) {$L(R_{n+1},\delta)$} ;
\node (Phin) at (0,2) {$\Phi_n(z)^{\uparrow}$} ;
\node (Phin+1) at (10,2) {$\Phi_{n+1}(z)^{\uparrow}$} ;

\draw [right hook-latex] (Phin) -- (Phin+1) node [midway,above] {$\iota_{n+1,n}$} ; 
\draw [->] (Phin) -- (Ln) node [sloped,midway,above] {$\sim$} node [midway,left] {$c_n$}; 
\draw [->] (Phin+1) -- (Ln+1) node [sloped,midway,below] {$\sim$} node [midway,right] {$c_{n+1}$} ;
\draw [->] (Ln) -- (Lsoc) node [midway,above] {$\sim$} node [midway,below] {$q^{\ast}$} ;
\draw [->] (Phin) -- (Lsoc) node [sloped,midway,below] {$\sim$} node [sloped,midway,above] {$f$} ;
\draw [right hook-latex] (Lsoc) -- (Ln+1) node [midway,below] {$p^{\ast}$} ;
\draw [->,bend right] (Ln) to [in=200,out=340] node [below] {$\varphi_{n+1,n}^{\ast}$} (Ln+1) ;
\end{tikzpicture}
\end{center}

Therefore, there is an isomorphism of lattices $L(R_n,\delta) \overset{\sim}{\to} L(S_n,\delta)$ that fits into this commutative diagram. By a result of Camillo \cite[Corollary 6.3)]{camillo},
this isomorphism is of the form $q^{\ast}$, where $q: S_n^{\delta} \to R_n^{\delta}$ is a $\gamma$-semilinear bijection with respect to some ring isomorphism $\gamma: R_n \to S_n$.

With the surjective homomorphism $\alpha_{n+1,n} = \gamma \circ \beta$, the map $\varphi_{n+1,n} = q \circ p$ is clearly an $\alpha$-semilinear surjection such that $\varphi_{n+1,n}^{\ast}$ makes the diagram commute!

By the surjective maps $\alpha_{n+1,n}$, the rings $R_n$ can be arranged in an inverse system of cpu rings. As $R_n$ is of length $n$, the ring $R = \lim_{\longleftarrow} R_n$ is a dvr by \cref{lem:inverse_limits_of_cpu_rings_are_dvrs}. By the canonical projection $R \to R_n$, each $R_n^{\delta}$ can actually be considered as an $R$-module.

Furthermore,
\[
M := \underset{\leftarrow}{\lim}\ R_n^{\delta} = \left\{ (m_n)_{n \geq 1} \in \prod_{n=1}^{\infty} R_n^{\delta} \, : \, \forall n \geq 1: \varphi_{n+1,n}(m_{n+1}) = m_n \right\}
\]
is naturally a left $R$-module with scalar multiplication $(r_n)_{n \geq 1} \cdot (m_n)_{n \geq 1} = (r_n m_n)_{n \geq 1}$.

We now show that ${}_RM \cong {}_RR^{\delta}$ (this is not completely trivial since the connecting maps between the different $R_n^{\delta}$ are not a priori given by reduction of coefficients). There are clearly elements $x^1,\ldots,x^{\delta} \in R_1^{\delta}$ which form a basis of $R_1^{\delta}$ (a cpu ring of length $1$ is a field!). We now choose $m^1,\ldots,m^{\delta} \in M$ with $m_1^i=x^i$ ($1 \leq i \leq \delta$). Our aim is now to show that $m^1, \ldots, m^{\delta}$ form a basis of ${}_RM$.

Fix some integer $n \geq 1$. We clearly have $\ker(\varphi_{n,1}) = \pi_nR_n^{\delta} = \Rad(R_n^{\delta})$, so $\varphi_{n,1}: R_n^{\delta} \to R_1^{\delta}$ provides an isomorphism of $R_n$-modules $\overline{\varphi}_{n,1}: R_n^{\delta}/\Rad(R_n^{\delta}) \overset{\sim}{\rightarrow} R_1^{\delta}$. Here we see $R_1^{\delta}$ as an $R_n$-module via restriction of scalars by $\alpha_{n,1}: R_n \to R_1$. Since the elements $\varphi_{n,1}(m_n^i) = m_1^i$ ($1 \leq i \leq \delta$) generate $R_1^{\delta}$, the elements $m_n^i$ ($1 \leq i \leq \delta$) generate $R_n^{\delta}$, by Nakayama's lemma. The module $R_n^{\delta}$ has composition length $\delta \cdot n$, so the $m_n^i$ ($1 \leq i \leq \delta$) must form a basis of $R_n^{\delta}$ for any fixed $k$.

Given any element $m =(m_n)_{n \geq 1} \in M$, we can therefore find for each $n \geq 1$ unique elements $r_n^i \in R_n$ ($1 \leq i \leq \delta$) such that $m_n = \sum_{i=1}^{\delta} r_n^i m_n^i$. For each pair of integers $k \geq l \geq 1$, we then have
\[
m_l = \varphi_{k,l} m_k
\]
\[
\Rightarrow  \sum_{i=1}^{\delta} r_l^i m_l^i = \varphi_{k,l} \left( \sum_{i=1}^{\delta} r_k^i m_k^i \right)
 = \sum_{i=1}^{\delta} \alpha_{k,l} (r_k^i) \varphi_{k,l}(m_k^i)
 = \sum_{i=1}^{\delta} \alpha_{k,l} (r_k^i) m_k^l.
\]
This proves that $\alpha_{k,l}(r_k^i)=r_l^i$. Therefore, for each $1 \leq i \leq \delta$, setting $r^i := (r_n^i)_{n \geq 1}$ defines an element of $R$. These $r^i \in R$ are the unique elements in $R$ with the property that $\sum_{i=1}^{\delta}r^im^i = m$. Since this shows that every element in $M$ is uniquely expressible as an $R$-linear combination of the elements $m^1, \ldots, m^{\delta}$, we have ${}_RM \cong {}_RR^{\delta}$.

Let $\pi \in R$ be a uniformizer. By restriction of scalars, we can make for $n \geq 1$ the identification $L_n = L({}_RM_n)$ where $M_n := M / \pi^n M$. We have seen that ${}_RM \cong {}_RR^{\delta}$. From \cref{prop:lat_r_delta_as_a_direct_limit}, we can therefore conclude that $G^- \cong \underset{\rightarrow}{\lim} \ L_k \cong \Lat({}_RR^{\delta})$.
\end{proof}

We can not only coordinatize semi-beams. We can even coordinatize whole beams!

\begin{thm} \label{thm:large_enough_beams_are_coordinatizable}
Let $z \in X(\Ccal(G^-))$ with $\delta_z := \deg(z) \geq 4$, then there is a noncommutative dvf $Q_z$ with dvr $R_z$ such that there is a lattice isomorphism
\[
\beth(z) \cong \Lat({}_{R_z}Q_z^{\delta_z})
\]
\end{thm}

\begin{proof}
As in the preceding proof, we suppress the $z$ in $R_z, Q_z, \ldots$. Let $c: \beth(z)^- \overset{\sim}{\to} \Lat(R,\delta)$ be the lattice isomorphism constructed in \cref{thm:large_enough_semibeams_are_coordinatizable}. We claim that this lattice isomorphism can uniquely be extended to a lattice isomorphism $\tilde{c} : \beth(z) \overset{\sim}{\to} \Lat({}_RQ^{\delta})$. 

The uniqueness is clear: if $\tilde{c}(g)$ is defined for some $g \in \beth(z)$, then, clearly
\[
\tilde{c} \left( \Soc_{\beth(z)}(g) \right) = \Soc_{\Lat({}_RQ^{\delta})} \left( \tilde{c}(g) \right).
\]
This argument essentially contains the construction of the extension:

Let $g \in \beth(z)$. We define
\[
\tilde{c}(g) := \Soc_{\Lat({}_RQ^{\delta})}^k \left( c \left( \Rad_{\beth(z)}^k(g) \right) \right),
\]
where $k \geq 0$ is an arbitrary integer with $\Rad_{\beth(z)}^k(g) \in \beth^-(z)$. This is indeed well-defined: letting $k_0 \geq 0$ be minimal with $\Rad_{\beth(z)}^{k_0}(g) \in \beth^-(z)$, we write any admissible $k$ in this definition as $k = k_0 + l$. Then
\begin{align*}
\Soc_{\Lat({}_RQ^{\delta})}^k(\Rad_{\beth(z)}^k(g)) & = \left( \Soc_{\Lat({}_RQ^{\delta})}^{k_0} \circ \Soc_{\Lat({}_RQ^{\delta})}^l \right) \left( c \left( \left( \Rad_{\beth(z)}^l  \circ \Rad_{\beth(z)}^{k_0} \right) (g) \right) \right) \\
& = \left( \Soc_{\Lat({}_RQ^{\delta})}^{k_0} \circ \Soc_{\Lat(R,\delta)}^l \right) \left( c \left( \left( \Rad_{\beth(z)^-}^l  \circ \Rad_{\beth(z)}^{k_0} \right) (g) \right) \right) \\
& = \left( \Soc_{\Lat({}_RQ^{\delta})}^{k_0} \circ \Soc_{\Lat(R,\delta)}^l  \circ \Rad_{\Lat(R,\delta)}^l \right) \left( c \left( \Rad_{\beth(z)}^{k_0}(g) \right) \right) \\
& = \Soc_{\Lat({}_RQ^{\delta})}^{k_0} \left( c \left( \Rad_{\beth(z)}^{k_0}(g) \right) \right).
\end{align*}

Note that for any $k \geq 0$, $\tilde{c}$ restricts to $\Soc^k(e)^{\downarrow}$ and $\Soc^k(R^{\delta})^{\downarrow}$ \footnote{Note that $R^{\delta}$ is the top element of $\Lat(R,\delta)$} as $\Soc_{\Lat({}_RQ^{\delta})}^k \circ c \circ \Rad_{\beth(z)}^k$, a composition of lattice isomorphisms (\cref{lem:soc_and_rad_are_inverse_on_beams}, \cref{lem:soc_and_rad_are_inverse_in_lat})). Therefore, $\tilde{c}$ is an isomorphism between $\beth(z)$ and $\Lat({}_RQ^{\delta})$.
\end{proof}

\section{Summary}

By what we have shown, every modular noetherian right $\ell$-group $G$ with strong order unit decomposes as a lattice-theoretic product $G = \prod_{i = 1}^k \beth(z_i)$, where $\{z_1,\ldots,z_k \} = X(\Ccal(G^-))$ (\cref{thm:G_is_product_of_beams}).

We have shown that a beam $\beth(z)$ with $\delta_z = \deg(z) \geq 4$ can furthermore be coordinatized by a a lattice of the form $\Lat({}_{R_z}Q_z^{\delta_z})$ where $Q_z$ is a dvf, dependent on $z$, with respective dvr $R_z$ (\cref{thm:large_enough_beams_are_coordinatizable}).

We give an outline how to derive a partial group-theoretic description of $G$ by these methods. A more detailed determination of the automorphism groups of $\Lat(R,\delta_z)$, $\Lat({}_RQ^{\delta_z})$ can be found in \cite{dietzel_dissertation}.

Let $\pi_z$ be a uniformizer of $R_z$ and set $R_{z,n} = R_z / \pi_z^nR_z$. Then $[\Phi_n(z),e] \cong L(R_{z,n},\delta_z)$ by \cref{prop:piecewise_coordinatization}.

As $\Phi_n(z) = \Rad_{\beth(z)}^n(e)$, any automorphism $\varphi \in \mathrm{Aut}(\beth^-(z))$ restricts to a series of automorphisms $(\varphi_n)_{n \geq 1}$ of all intervals $[\Phi_n(z),e] \cong L(R_{z,n},\delta_z)$. By Camillo's result \cite[Corollary 6.3.]{camillo}, these automorphisms come from a series of bijections $f_n : R_{z,n}^{\delta_z} \to R_{z,n}^{\delta_z}$, that are $\alpha_n$-semilinear with respect to some ring automorphisms $\alpha_n \in \mathrm{Aut}(R_{z,n})$.

Two $\alpha$-semilinear bijections $f_n,g_n: R_{z,n}^{\delta_z} \to R_{z,n}^{\delta_z}$ induce the same element of $\mathrm{Aut}(L(R_{z,n},\delta_z))$ if and only if there is a unit $\gamma \in R_{z,n}^{\times}$, such that $g_n = \gamma \cdot f_n$. This implies that
\[
\mathrm{Aut}(L(R_{z,n},\delta_z)) \cong \left( \mathrm{GL}(R_{z,n},\delta_z) \rtimes \mathrm{Aut}(R_{z,n}) \right) / R_{z,n}^{\times} =: \mathrm{P}\Gamma\mathrm{L}(R_{z,n},\delta_z).
\]
Where $\mathrm{GL}(R,\delta_z) = \mathrm{Aut}({}_RR^{\delta_z})$, the automorphism group of a free $R$-module of rank $\delta_z$. Taking inverse limits, one proves that
\[
\mathrm{Aut}(\Lat(R_z,\delta_z)) \cong \left( \mathrm{GL}(R_z,\delta_z) \rtimes \mathrm{Aut}(R_z)) \right) / R_z^{\times} =: \mathrm{P} \Gamma \mathrm{L}(R_z,\delta_z).
\]
Now, each element $f \in \Gamma \mathrm{L}(R_z,\delta_z)$ can uniquely be extended to a semilinear automorphism of ${}_{R_z}Q_z^{\delta_z}$. If $f: {}_{R_z}R_z^{\delta_z} \to {}_{R_z}R_z^{\delta_z}$ is $\alpha$-semilinear, the respective extension to $Q_z^{\delta_z}$ is given by $\overline{f}(\pi^{-n} \cdot x) = \alpha(\pi)^{-n} \cdot f(x)$, where $n \geq 0$ and $x \in R_z^{\delta_z}$. We furthermore note that there is only one way of extending automorphisms of $\Lat(R_z,\delta_z)$ to $\Lat({}_{R_z}Q_z^{\delta_z})$, as (lattice-theoretical) socles must be respected by automorphisms.

Therefore, we have an embedding
\[
\mathrm{P} \Gamma \mathrm{L}(R_z,\delta_z) \leq \mathrm{P} \Gamma \mathrm{L}({}_{R_z}Q_z^{\delta_z}) := \Gamma \mathrm{L}({}_{R_z}Q_z^{\delta_z}) / R_z^{\times}.
\]
This is exactly the embedding of the stabilizer of $R_z^{\delta_z}$ under the action of $\mathrm{P} \Gamma \mathrm{L}({}_{R_z}Q_z^{\delta_z})$ on $\Lat({}_{R_z}Q_z^{\delta_z})$. As $\mathrm{GL}(Q_z,\delta_z)$ acts transitively on $\Lat({}_{R_z}Q_z^{\delta_z})$, every automorphism of $\Lat({}_{R_z}Q_z^{\delta_z})$ is induced by the composition of an element of $\Gamma \mathrm{L}(R_z,\delta_z)$ and an element of $\mathrm{GL}(Q_z,\delta_z)$, which proves that $\mathrm{Aut}(\Lat({}_{R_z}Q_z^{\delta_z})) \cong \mathrm{P} \Gamma \mathrm{L}({}_{R_z}Q_z^{\delta_z})$.

We have demonstrated at the end of \cref{sec:the_distributive_scaffold} that if $G_I = \prod_{i \in I} \beth(z_i) $ is an isotypical component, then $G_I$ is a sublattice of $G$ that is fixed setwise under right-multiplication by elements of $G_I$, so it is a right $\ell$-group on its own. By the rigidity for beams (\cref{thm:rigidity_of_beams}), we conclude that $G_I$ can be identified with a subgroup of the wreath product $\mathrm{Sym}(I) \wr \mathrm{Aut}(\beth(z_i))$ ($i \in I$ arbitrary), that is a complement to the subgroup $\mathrm{Sym}(I) \wr \mathrm{Aut}(\beth(z_i)^-)$.

We conclude that an isotypical component $G_I$ of a beam $\beth(z_i)$ with $\deg(z_i) \geq 4$ can be identified with a complement of the subgroup
\[
\mathrm{Sym}(I) \wr \mathrm{P}\Gamma \mathrm{L}(R_I,\delta_z) \leq \mathrm{Sym}(I) \wr \mathrm{P}\Gamma \mathrm{L}({}_{R_I}{Q_I}^{\delta_z}),
\]
where $Q_I$ is a dvf with dvr $R_I$ that is only dependent on the component $G_I$.

It is not known to the author if there is a nice description for the isotypical components of a beam $\beth(z_i)$ with $\deg(z_i) = 3$, as the coordinatization methods used in this article do not work anymore in this case. All examples that are known to the author are in fact coordinatizable but it seems that we are far from knowing if there are any non-coordinatizable beams of dimension $3$. This is somewhat reminiscent of the long-standing open geometrical problem if a finite transitive plane is desarguesian; this may or may not give an impression of the hardness of this question\ldots

\section*{Acknowledgements and funding}

The author of this article is recently granted a Feodor Lynen fellowship by the Alexander Humboldt Foundation, which enables him to do his research in the hospitable and inspiring environment of the research group \emph{Algebra \& Analysis} at the \emph{Vrije Universiteit Brussel}. This work was partially supported by the project OZR3762 of Vrije Universiteit Brussel.

\bibliography{ortho}

\bibliographystyle{halpha}

\end{document}